\documentclass[12pt, reqno]{amsart}

\usepackage{enumitem}
\usepackage{mathrsfs}
\usepackage{hyperref}
\usepackage{amsthm}
\usepackage{MnSymbol}
\usepackage{bbold}

\usepackage{chngcntr}
\counterwithin*{equation}{section}
\counterwithin*{equation}{subsection}
\usepackage{stmaryrd}
\usepackage{fullpage}
\usepackage{parskip}

\usepackage[utf8]{inputenc}
\usepackage{csquotes}
\usepackage{cooltooltips}
\usepackage[english]{babel}

\usepackage{mathtools}

\usepackage[T1]{fontenc}
\usepackage{textcomp}
\usepackage{bm}

\usepackage{color}

\newtheorem{theorem}{Theorem}[section]
\newtheorem{corollary}[theorem]{Corollary}
\newtheorem{lemma}[theorem]{Lemma}
\newtheorem{definition}[theorem]{Definition}
\newtheorem{definitionlemma}[theorem]{Definition-Lemma}
\newtheorem{proposition}[theorem]{Proposition}
\newtheorem{remark}[theorem]{Remark}

\let\originalmiddle=\middle
\renewcommand{\middle}[1]{\mathrel{}\originalmiddle#1\mathrel{}}

\newcommand\restr[2]{{
    \left.\kern-\nulldelimiterspace 
    #1 
    \vphantom{\big|} 
    \right|_{#2} 
}}

\DeclareMathOperator{\lcm}{lcm}
\DeclareMathOperator{\Aut}{Aut}
\DeclareMathOperator{\Id}{Id}
\DeclareMathOperator{\Span}{Span}

\DeclareMathOperator{\centralprod}{\bigcirc}

\begin{document}

\title{On the automorphism group of the monoid of the integers modulo a prime power}
\author{Joseph Atalaye}
\address{
    Department of Mathematical Sciences,
    University of Stellenbosch,
    Stellenbosch,
    South Africa
}
\email{26828146@sun.ac.za}

\author{Liam Baker}
\address{
    Department of Mathematical Sciences,
    University of Stellenbosch,
    Stellenbosch,
    South Africa \\
    \&
    NITheCS (National Institute for Theoretical and Computational Sciences),
    South Africa
}
\email{liambaker@sun.ac.za}
\urladdr{https://math.sun.ac.za/liambaker}

\author{Sophie Marques}
\address{
    Department of Mathematical Sciences,
    University of Stellenbosch,
    Stellenbosch,
    South Africa \\
    \&
    NITheCS (National Institute for Theoretical and Computational Sciences),
    South Africa
}
\email{smarques@sun.ac.za}
\urladdr{https://sites.google.com/site/sophiemarques64/}

\begin{abstract}
This paper determines the structure of the automorphism group of the unit group \((U_{p^e}, \cdot)\) and the monoid \((\mathbb{Z}/p^e \mathbb{Z}, \cdot)\).
For \( e \geq 5 \), we establish that the automorphism group \( \Aut(U_{2^e}, \cdot) \) is the direct product of \( \mathbb{Z}/2\mathbb{Z} \) with the central product of a dihedral group of order 8 and the cyclic group \( \mathbb{Z}/2^{e-3}\mathbb{Z} \).
Moreover, we show that the automorphism group \( \Aut(\mathbb{Z}/p^e \mathbb{Z}, \cdot) \) is isomorphic to a canonical semidirect product of \( U_{p^{e-1}} \) and the subgroup of \( \Aut(U_{p^e}, \cdot) \) consisting of automorphisms that induce an automorphism of \( (U_{p^f}, \cdot) \) for any integer \( f \) such that \( 0 \leq f \leq e \).
\end{abstract}

\maketitle

\setcounter{tocdepth}{3}
\tableofcontents

\section*{Acknowledgement}
The authors are grateful to the anonymous referee for their valuable corrections and insightful suggestions, which have greatly improved this paper.


\section*{Introduction}

Studying the automorphisms of a structure is undoubtedly relevant in mathematics for several important reasons: they reveal the symmetries of mathematical structures, aid in classifying them, and preserve their intrinsic properties (see, for instance, \cite{zbMATH01650459}, \cite{zbMATH03549232},\cite{zbMATH06968189}).

We consider a prime $p$ and a positive integer $e$.
Computing the automorphism group of \((\mathbb{Z}/p^e \mathbb{Z},+)\) is a fundamental exercise in basic algebra courses.
However, an equally relevant question arises:
\begin{center}
    \textsf{What is the automorphism group of the commutative monoid \((\mathbb{Z}/p^e \mathbb{Z}, \cdot)\)?}
\end{center}
(We refer to Definition \ref{def:monoid} to establish the concept of a monoid.)\\
In this paper, we determine the structure of the automorphism group of this monoid.
In the process, we also clarify the structure of the automorphism group of its group of units, \((U_{p^e}, \cdot)\).
The primary challenge lies in the case where \(p = 2\), as the structure for odd \(p\) is relatively straightforward.

The study of these automorphisms is further motivated by the observation that any multiplicative automorphism \( \phi \) defines a new addition \( +_\phi \) on \( \mathbb{Z}/p^e \mathbb{Z} \), given by
\[
a +_\phi b = \phi^{-1}(\phi(a) + \phi(b))
\]
for all \( a, b \in \mathbb{Z}/p^e \mathbb{Z} \).
This operation gives a novel ring structure: \( (\mathbb{Z}/p^e \mathbb{Z}, +_\phi, \cdot) \).

Furthermore, the automorphism \( \phi \) induces an action \( \cdot_\phi \) of the ring \( (\mathbb{Z}/p^e \mathbb{Z}, +_\phi, \cdot) \) on \( (\mathbb{Z}/p^e \mathbb{Z}, +) \), defined by
\[
a \cdot_\phi b = \phi(a) \cdot b.
\]
This induced action provides a useful framework for analyzing the structural properties of \( \mathbb{Z}/p^e \mathbb{Z} \).
In particular, we can interpret \( (\mathbb{Z}/p^e \mathbb{Z}, +, \cdot) \) as an algebra over the ring \( (\mathbb{Z}/p^e \mathbb{Z}, +_\phi, \cdot) \), via the isomorphism
\[
\phi: (\mathbb{Z}/p^e \mathbb{Z}, +_\phi, \cdot) \to (\mathbb{Z}/p^e \mathbb{Z}, +, \cdot).
\]

The main theorem of this paper, Theorem \ref{Theo3.3}, establishes that \(\Aut \left( \mathbb{Z}/p^e \mathbb{Z}, \cdot \right)\) is isomorphic to a canonical semidirect product of \(U_{p^{e-1}}\) and the subgroup of \( \Aut(U_{p^e}, \cdot) \) consisting of automorphisms that induce an automorphism of \( (U_{p^f}, \cdot) \) for any integer \( f \) such that \( 0 \leq f \leq e \).
Additionally, in Theorem \ref{Theo2.6}, we establish that for \( e \geq 5 \), the automorphism group \( \Aut(U_{2^e}, \cdot) \) is isomorphic to the direct product of \( \mathbb{Z}/2\mathbb{Z} \) and the central product of the dihedral group of order 8 with the cyclic group \( \mathbb{Z}/2^{e-3}\mathbb{Z} \).
Furthermore, in Corollary \ref{strape}, we describe \( \Aut \left( \mathbb{Z}/p^e \mathbb{Z}, \cdot \right) \) in terms of (semi-)direct products involving well-known groups.

The paper is structured as follows:

In the first section, we recall some general properties of monoid morphisms that will be useful for the remaining part of the paper.

The second section delves into the automorphisms of the unit group \( (U_{p^e},\cdot) \), as understanding these is key to characterizing the automorphisms of the monoid \((\mathbb{Z}/p^e\mathbb{Z},\cdot)\), which is addressed in Section 3.
The analysis focuses primarily on the case \( p = 2 \), as the case for odd \( p \) is relatively straightforward.
After establishing some basic results about the group \((U_{p^e} ,\cdot)\), including computing the order of its elements (see Lemma \ref{Lem2.1}), comparing the properties of the group of units for different powers of \(p\) (see Lemma \ref{Lem2.3}), and determining all the minimal cardinality generator systems for \((U_{p^e},\cdot)\) (see Lemma \ref{Lem2.3}), we compute the automorphism group of \((U_{p^e},\cdot)\) (see Proposition \ref{Prop2.3}) and identify its structure as a known group construction (see Theorem \ref{Theo2.6}).

The third section proceeds similarly but with a focus on \((\mathbb{Z}/p^e\mathbb{Z},\cdot)\).
We establish some basic results about the monoid \((\mathbb{Z}/p^e \mathbb{Z}, \cdot)\), including criteria for determining when two elements are equal in \(\mathbb{Z}/p^e \mathbb{Z}\) (see Lemma \ref{Lem3.1}) and identifying all the minimal sets of generators for \((\mathbb{Z}/p^e \mathbb{Z} ,\cdot)\) (see Lemma \ref{lem:ZpeZ_generators}).
From these results, we successfully determine the structure of the automorphism group of \((\mathbb{Z}/p^e \mathbb{Z}, \cdot)\).
Specifically, for \(p^e \neq 8\), we identify the automorphism group as a canonical semidirect product of \(U_{p^{e-1}}\) and the subgroup of \( \Aut(U_{p^e}, \cdot) \) consisting of automorphisms that induce an automorphism of \( (U_{p^f}, \cdot) \) for any integer \( f \) such that \( 0 \leq f \leq e \) (see Theorem \ref{Theo3.3}) and identify its structure as a known group construction (see Corollary \ref{strape}).

\section{Preliminaries}
For this paper, \(\mathbb{N}\) denotes the set of natural numbers, including \(0\), \(p\) denotes a prime number and \(e\) denotes a positive integer.
We adopt the convention that \(0^0 = 1\).
We denote by \([a]_n\) the equivalence class of an integer \(a\) modulo \(n\).
We will simply write \([a]\) for \([a]_{p^e}\).
For integers $a$, $b$, and $n$ we write $a^n \mid\mid b$ if $a^n \mid b$ and $a^{n+1} \nmid b$, and for $a < b$ we use the notation $\llbracket a,b \rrbracket = \{a, a+1, \dotsc, b\}$.

The purpose of this paper is to determine the automorphisms of \((U_{p^e},\cdot)\), the group of integers coprime to \(p\) modulo \(p^e\), and \((\mathbb{Z}/p^e\mathbb{Z},\cdot)\), the multiplicative monoid of integers modulo $p^e$.
We start the section with the definition of a monoid.
\begin{definition} \label{def:monoid}
A system \((G, \cdot, 1)\) is called a \textsf{monoid} if \(\cdot\) is an associative binary operation on \(G\) and \(1\) is a neutral element, meaning that for all \( g \in G \), we have \( g \cdot 1 = 1 \cdot g = g \).
The monoid is said to be \textsf{commutative} if \( gh = hg \) for all \( g, h \in G \).
We may also denote the monoid as \((G, \cdot)\), or simply as \(G\) when the operation \(\cdot\) is clear from the context.
\end{definition}

The following definitions introduce various fundamental concepts related to monoids, which we will use throughout the paper.

\begin{definition}\label{Def1.1}
Let $(G, \cdot, 1)$ be a monoid, and let $S$ be a subset of $G$.
\begin{enumerate}
    \item We denote by $G^\times$ the \textsf{set of invertible elements} in $G$, defined as \[ \left\{ g \in G \middle| \exists h \in G \ (gh = hg = 1) \right\}. \]
    We note that $(G^\times, \cdot, 1)$ forms a group.
    \item We define $\Span(S)$ as the intersection of all submonoids of $(G, \cdot, 1)$ with respect to $\cdot$ that contain the set $S$. $\Span(S)$ is itself a monoid.
    \item We say that $S$ generates $G$ if $G = \Span(S)$.
    \item We denote by $\Aut(G, \cdot)$ the \textsf{automorphism group of the monoid} $(G, \cdot, 1)$.
\end{enumerate}
\end{definition}

The following lemma identifies a few elementary but useful properties of automorphisms of a monoid, and its proof is straightforward.

\begin{lemma}\label{Lem1.2}
Let \((G, \cdot, 1)\) be a monoid and let \(\phi \in \Aut(G, \cdot)\).
Then:
\begin{enumerate}
    \item \(\phi(G^\times) = G^\times\), and the map \(\tilde{\phi}: G^\times \rightarrow G^\times\), \(x \mapsto \phi(x)\) is a group automorphism.
    \item If \(G\) contains an absorbing element \(0\) (that is, \(0 \in G\) and \(0 \cdot x = x \cdot 0 = 0\) for all \(x \in G\)), then such an element is unique and \(\phi(0) = 0\).
\end{enumerate}
\end{lemma}

\section{Characterizing the automorphism group of the unit group modulo \texorpdfstring{$p^e$}{}}
\subsection{Basic results about the unit group modulo \texorpdfstring{$p^e$}{}}
It is well known (see, for instance, \cite[Theorem 6.10]{zbMATH01115328}) that when \( p \) is odd, \( U_{p^e} \) is a cyclic group of order \( p^{e-1}(p-1) \).
On the other hand, when \( p = 2 \) and \( e \geq 3 \), \( U_{2^e} \) is the direct product of a cyclic group of order \( 2^{e-2} \) (generated by \([5]\)) and a cyclic group of order \( 2 \) (generated by \([-1]\)).
This fundamental result will be used throughout the paper.

We begin this section by recalling the method for computing the order of an element in \( U_{2^e} \).
\begin{lemma}\label{Lem2.1}
Let \(e \geq 2\), \(v \in \{0, 1\}\), and \(w \in \llbracket 0, 2^{e-2}-1 \rrbracket\).
For \(w \neq 0\), let \(\beta\) denote the largest integer such that \(2^{\beta} \mid\mid w\) (i.e., \(2^{\beta}\) divides \(w\) but \(2^{\beta+1}\) does not).
Then, the order of the element \([-1]^v[5]^w\) in \(U_{2^e}\) is given by
\[
    o([-1]^v[5]^w) =
    \begin{cases}
        1 & \text{if } (v,w) = (0,0), \\
        2 & \text{if } (v,w) = (1,0), \\
        2^{e-2-\beta} & \text{if } w \neq 0.
    \end{cases}
\]
\end{lemma}

\begin{proof}
The formula can be derived from the following identities, which hold in any abelian group such as \( U_{2^e} \):

\begin{itemize}
    \item \( o([-1]^v[5]^w) = \lcm(o([-1]^v), o([5]^w)) \) (see \cite[Lemma 9.5.12]{Carstensen2019});
    \item \( o([5]^w) = \frac{2^{e-2}}{\gcd(w, 2^{e-2})} \) (see \cite[Theorem 0.2.11]{roman2005field});
    \item
    \[
    o([-1]^v) =
    \begin{cases}
        1, & \text{if } v = 0, \\
        2, & \text{if } v = 1.
    \end{cases}
    \]
\end{itemize}

Combining these, the result follows.
\end{proof}

\begin{lemma}\label{Lem5.14}
For $e\geq 5$, we have $[5^{w+2^{e-5}}]_{2^{e-2}} = [5^{w} +2^{e-3} ]_{2^{e-2}}$.
\end{lemma}
\begin{proof}
We have
\begin{align*}
    [5^{ 2^{e-5}+w}]_{2^{e-2}} &= [5^w(5^{ 2^{e-5}})]_{2^{e-2}}=[5^w(1+2^{e-3})]_{2^{e-2}}, & &\text{by proof of \cite[Theorem 6.7]{zbMATH01115328}} \\
    &= [5^w+5^w.2^{e-3}]_{2^{e-2}}=[5^w+2^{e-3}]_{2^{e-2}}, & &\text{since $5$ is odd.} \qedhere
\end{align*}
\end{proof}

In the following lemma, we identify the minimal set of generators for \((U_{p^e},\cdot)\).
\begin{lemma} \label{Lem2.3} ~
\begin{enumerate}
    \item When \( p \) is odd, \( U_{p^e} \) is cyclic of order \( p^{e-1}(p-1) \).
    More precisely, for any generator \( g \) of \( U_{p^e} \), any generator of \( U_{p^e} \) can be written as \( g^v \), where \( v \in \mathbb{Z} \) and \( \gcd(v, p^{e-1}(p-1)) = 1 \).

    \item When \( e \geq 3 \), the minimal number of generators for \( U_{2^e} \) is 2.
    More precisely, any set of generators of \( U_{2^e} \) of size 2 is of the form \( \{ [-1][5]^u, [-1]^v[5]^w \} \), where
    \[
    u, w \in \llbracket 0, 2^{e-2} - 1 \rrbracket, \quad \gcd(w, 2) = 1, \quad v \in \{0,1\}, \quad \text{and} \quad 2 \mid u \text{ when } v = 1.
    \]
\end{enumerate}
\end{lemma}

\begin{proof} ~
\begin{enumerate}
    \item $U_{p^e}$ is cyclic of order $p^{e-1}(p-1)$ (see \cite[Theorem 6.7]{zbMATH01115328}).
    Any other generator of \( U_{p^e} \) must be of the form \( g^v \), where \( v \in \mathbb{Z} \) and \( \gcd(v, p^{e-1}(p-1)) = 1 \), ensuring that \( v \) is coprime to the order of \( g \).
    \item By \cite[Theorem 6.10]{zbMATH02016545}, \( U_{2^e} \) is non-cyclic and generated by \([-1]\) and \([5]\), thus the minimal cardinality for a set of generators of \( U_{2^e} \) is 2.
    We now compute all the sets of generators with cardinality $2$.
    Suppose that \(\{[-1]^{t}[5]^{u}, [-1]^{v}[5]^{w}\}\) forms a system of generators, $u,w \in \llbracket 0, 2^{e-2}-1 \rrbracket$ and $t, v \in \{ 0, 1\}$.
    Then there exist integers \(x, y, x', y'\) such that
    \[
        \left\{
        \begin{aligned}
            ([-1]^{t}[5]^{u})^{x}([-1]^{v}[5]^{w})^{y} &= [-1], \\
            ([-1]^{t}[5]^{u})^{x'}([-1]^{v}[5]^{w})^{y'} &= [5].
        \end{aligned}
        \right.
    \]
    Since \( U_{2^e} \) is the direct product of the group generated by \( [-1] \) and the group generated by \( [5] \), the previous equalities hold if and only if there exist integers \( x, y, x', y' \) satisfying the system:
    \begin{equation} \label{eq:E}
        \left\{
            \begin{aligned}
                [tx + vy]_2   &= [1]_2      & \text{(E1)}\ \\
                [ux + wy]_{2^{e-2}}  &= [0]_{2^{e-2}} & \text{(E2)}\ \\
                [tx' + vy']_2 &\equiv [0]_2       & \text{(E3)}\ \\
                [ux' + wy']_{2^{e-2}} &\equiv [1]_{2^{e-2}}  & \text{(E4)}.
            \end{aligned}
        \right.
        \tag{E}
    \end{equation}
If both \( u \) and \( w \) are even, equation (E4) leads to a contradiction.
Hence, without loss of generality, we assume that \( w \) is odd, due to the symmetry \( (t,u) \leftrightarrow (v,w) \) in the system.

When \( t = 0 \) and \( v = 0 \), equation (E1) is not solvable, which means the system has no solution in this case.

Now, solving the case \( t = 0 \) and \( v = 1 \) is equivalent to solving the case \( t = 1 \) and \( v = 0 \), thanks to the symmetry of the system.
Specifically, when \( t = 0 \) and \( v = 1 \), equation (E3) implies \( [y']_2= [0]_2 \), and equation (E4) implies \( [ux']_2 = [1]_2 \).
Thus, we deduce that \( [u]_2= [1]_2  \).
The symmetry \( (t,u) \leftrightarrow (v,w) \) is valid, as we assumed \( w \) to be odd.

As a result, we can assume without loss of generality that \( t = 1 \).

When \( v = 1 \), subtracting equation (E2) from equation \(w\)(E1) gives \( [(w - u)x]_2=[1]_2 \), as \( e \geq 3 \).
This implies \([ w - u]_2 =[1]_2 \), leading to \( 2 \mid u \) since \( w \) is odd.

Finally, when \( t = 1 \), and either \( v = 0 \) or \( (v = 1 \text{ and } 2 \mid u) \), the system (E) always has a solution.
Specifically, the solutions are:

\begin{itemize}
    \item \( (x, y, x', y') = (w, -u, 0, w') \) when \( v = 0 \),
    \item \( (x, y, x', y') = (w, -u, 1, w'(1 - u)) \) when \( v = 1 \) and \( 2 \mid u \),
\end{itemize}
where \( w' \) is the multiplicative inverse of \( w \) modulo \( 2^e \).
    \qedhere
\end{enumerate}
\end{proof}

\subsection{The structure of the automorphism  group of the unit group modulo \texorpdfstring{$p^e$}{}}
The following proposition establishes the structure of the automorphism group of \( (U_{p^e}, \cdot) \) for any prime power \( p^e \).
\begin{proposition} \label{Prop2.3} ~
\begin{enumerate}
    \item Let \( p \) be an odd prime.
    The automorphism group \(\Aut(U_{p^e}, \cdot)\) of the group \( (U_{p^e}, \cdot) \) is given by
    \[
    \Aut(U_{p^e}, \cdot) = \left\{ \chi_{p^e, t} \mid t \in \llbracket 0, p^{e-1}(p-1)-1 \rrbracket, \, \gcd(t, p^{e-1}(p-1)) = 1 \right\},
    \]
    where for any \( t \) with \(\gcd(t, p^{e-1}(p-1)) = 1 \), the map \(\chi_{p^e, t}\) is defined by \( \chi_{p^e, t}(a) = a^t \), for all $a\in U_{p^e}$.

    Moreover, the order of the automorphism group is
    \[
    \left| \Aut(U_{p^e}, \cdot) \right| = \varphi(p^{e-1}(p-1)),
    \]
    where \(\varphi\) denotes Euler's totient function.

    \item For the case where $p=2$, we consider the following cases:
    \begin{enumerate}
        \item When \( e \in \{1, 2\} \), we have \( \Aut(U_{2^e}, \cdot) = \{\Id\} \), where \(\Id\) denotes the identity automorphism.

        \item When \( e = 3 \), the automorphism group is:
        \[
            \Aut(U_{8}, \cdot) = \left\{ \chi_{8,\sigma} \middle| \sigma \in S_3 \right\},
        \]
        where for any \( \sigma \in S_3 \), the automorphism \(\chi_{8,\sigma}\) is defined as:
        \[
            \phi_{\sigma}(a) =
            \begin{cases}
                [1] & \text{if } a = [1], \\
                h(\sigma(h^{-1}(a))) & \text{if } a \neq [1],
            \end{cases}
        \]
        with \( h \) being the map from \(\{1, 2, 3\}\) to \(\{[-1], [5], [-5]\}\) such that \( h(1) = [-1] \), \( h(2) = [5] \), and \( h(3) = [-5] \).
        In particular, \[ \left| \Aut(U_{8}, \cdot) \right| = 6 .\]

        \item When \( e \geq 4 \), the automorphism group is given by:
        \[
            \Aut(U_{2^e}, \cdot) = \left\{ \chi_{2^e,t} \middle| t \in T_{2^e} \right\},
        \]
        where
        \[
           \ \ \ \ \ \ T_{2^e} = \left\{ (t_1, t_2, t_3) \middle| t_1 \in \{0, 2^{e-3}\}, t_2 \in \{0, 1\}, t_3 \in \llbracket 0, 2^{e-2}-1 \rrbracket, \gcd(t_3, 2) = 1 \right\},
        \]
        and, for each \( t = (t_1, t_2, t_3) \in T_{2^e} \), the automorphism \(\chi_{2^e,t}\) maps \([-1]\) to \([-1][5]^{t_1}\) and \([5]\) to \([-1]^{t_2}[5]^{t_3}\).
        The size of the automorphism group is:
        \[
            |\Aut(U_{2^e}, \cdot)| = 2^{e-1}.
        \]
    \end{enumerate}
\end{enumerate}
\end{proposition}

\begin{proof}~
\begin{enumerate}
    \item This is clear since $U_{p^e}$ is cyclic.
    \item
    \begin{enumerate}
        \item This is clear.
        \item  The result follows immediately from the isomorphism between \( U_8 \) and the Klein group.
        For the isomorphism between \( U_8 \) and the Klein group, see \cite[Section 6.2]{zbMATH01115328}, and for the isomorphism between the automorphism group of the Klein group and \( S_3 \), refer to \cite[Exercise 5.3.7, p. 254]{goodman2014algebra}.

        \item We now aim to construct an automorphism \( \phi \) of \( U_{2^e} \) when $e \geq 4$.
        To do this, we observe that such an automorphism \(\phi\) must satisfy the following properties:
        \begin{enumerate}
            \item \(\phi\) sends the set of generators \(\{[-1], [5]\}\) to a set of generators of \(U_{2^e}\);
            \item \(\phi([-1])\) has order 2; and
            \item \(\phi([5])\) has order \(2^{e-2}\).
        \end{enumerate}

        Therefore, by Lemma \ref{Lem2.1} and Lemma \ref{Lem2.3}, we deduce that \(\phi([-1]) = [-1][5]^{t_1}\) and \(\phi([5]) = [-1]^{t_2}[5]^{t_3}\), where \((t_1, t_2, t_3) \in T_{2^e}\).

        Let us choose such a map \(\phi\).
        That is, we choose \((t_1, t_2, t_3) \in T_{2^e}\) and set \(\phi([-1]^v[5]^w) = [-1]^{v + t_2 w}[5]^{t_1 v + t_3 w}\) for all \(v, w \in \mathbb{Z}\).
        We now prove that such a \(\phi\) is well-defined.
        Let \(v, w \in \mathbb{Z}\) such that
        \begin{equation}\label{Eq12}
            [-1]^v[5]^w = [1].
        \end{equation}
        We prove that
        \begin{equation}\label{Eq13}
            [-1]^{v + t_2 w}[5]^{t_1 v + t_3 w} = [1].
        \end{equation}

    From the direct product structure of \( U_{2^e} \), we conclude that Equation \eqref{Eq12} is equivalent to \( [v]_2 = [0]_2 \) and \( [w]_{2^{e-2}} = [0]_{2^{e-2}} \).
  
    Therefore, since \( e \geq 4 \), it follows from Equation \eqref{Eq12} that \([v + t_2 w]_2 = [0]_2\).
    Given that \( t_1 \in \{0, 2^{e-3}\} \) and \([v]_2 = [0]_2\), we have \([t_1 v]_{2^{e-2}} = [0]_{2^{e-2}}\).
    Hence, \([t_1 v + t_3 w]_{2^{e-2}} = [0]_{2^{e-2}}\).
    From the direct product structure of \( U_{2^e} \), we can conclude that \(\phi\) is well-defined.
  
    Since \(\phi\) is surjective by construction, satisfying condition (i), it is well-defined.
    Additionally, as it maps a finite set onto itself, surjectivity implies bijectivity.
    Thus, \(\phi\) is an automorphism, as required. \qedhere
    \end{enumerate}
\end{enumerate}
\end{proof}

Given $0 \leq f \leq e$ and $a\in \mathbb{Z}$, it is clear that $[a]\in U_{p^e}$ implies $[a]_{p^f} \in U_{p^f}$.
Moreover,
\begin{itemize}
    \item when $p$ is odd, if $[a] \in U_{p^e}$ is a generator for $U_{p^e}$ then $[a]_{p^f}$ is a generator for $U_{p^f}$;
    \item  when $p$ is $2$ and $e\geq 3$, if $\lbrace [a], [b]\rbrace $ is a system of generators for $U_{2^e}$ then $\lbrace [a]_{2^f}, [b]_{2^f}\rbrace$ is a system of generators for $U_{2^f}$.
\end{itemize}
We now define a subgroup of the automorphism group of \((U_{p^e}, \cdot)\), which will turn out to be precisely the set of automorphisms of \((U_{p^e}, \cdot)\) that can be lifted to automorphisms of \((\mathbb{Z}/p^e\mathbb{Z}, \cdot)\).
\begin{definitionlemma}\label{Deflem2.5} ~
\begin{enumerate}
    \item We define the subgroup \( \mathcal{A}_{p^e} \) of \( \Aut(U_{p^e}, \cdot) \) as the set of automorphisms of \( U_{p^e} \) that induce an automorphism of \( (U_{p^f}, \cdot) \) for every integer \( f \) satisfying \( 0 \leq f \leq e \).
    In other words, $\phi$ is an element of \( \mathcal{A}_{p^e} \), if for all \( r, r', s, s', f \in \mathbb{Z} \) with \( \gcd(r, p) = \gcd(r', p) = 1 \), \( \phi([r]) = [s] \), \( \phi([r']) = [s'] \), and \( 0 \leq f \leq e \), the following condition holds:
    \[
    [r]_{p^f} = [r']_{p^f} \quad \text{if and only if} \quad [s]_{p^f} = [s']_{p^f}.
    \]

    \item For any \( \phi \in \mathcal{A}_{p^e} \) and any integer \( f \) such that \( 0 \leq f \leq e \), we denote by \( \phi_f \) the automorphism of \( U_{p^f} \) induced by \( \phi \).
    More specifically, the map \( \phi_f \) is defined as:
        \begin{align*}
        \phi_f: (U_{p^f}, \cdot) &\rightarrow (U_{p^f}, \cdot), \\
        [r]_{p^f} &\mapsto [s]_{p^f},
        \end{align*}
    where \( s \) is an integer such that \( \phi([r]) = [s] \).
\end{enumerate}
\end{definitionlemma}
\begin{proof}
By the definition of \( \mathcal{A}_{p^e} \), \( \phi_f \) is a well-defined bijection, as it is an injective map between two finite sets of the same cardinality.

It remains to show that \( \phi_f \) is a homomorphism.
Let \( [r]_{p^f} \) and \( [r']_{p^f} \) be elements in \( U_{p^f} \).
By the definition of \( \phi_f \), we have:
\[
\phi_f([r]_{p^f} [r']_{p^f}) = \phi_f([rr']_{p^f}) = [s'']_{p^f},
\]
where \( s'' \) is the integer such that \( \phi([rr']) = [s''] \).

On the other hand, we also have:
\[
\phi_f([r]_{p^f}) \phi_f([r']_{p^f}) = [s]_{p^f} [s']_{p^f},
\]
where \( s \) and \( s' \) are integers such that \( \phi([r]) = [s] \) and \( \phi([r']) = [s'] \).

Since \( \phi \) is a homomorphism by assumption, we know that:
\[
\phi([rr']) = \phi([r]) \phi([r']).
\]
Thus, we obtain:
\[
[s''] = [s] [s'] = [ss']
\]
from which it follows that
\[
[s'']_{p^f} = [ss']_{p^f} = [s]_{p^f} [s']_{p^f}.
\]
This completes the proof.
\end{proof}

\begin{corollary} \label{theta}
Let \( f \) be an integer such that \( 0 \leq f \leq e \).
The map
\[
\begin{array}{ccclc}
\Theta &: & (\mathcal{A}_{p^e}, \circ) & \to & (\mathcal{A}_{p^f}, \circ) \\
    & & \phi & \mapsto & \phi_f
\end{array}
\]
is a well-defined group homomorphism.

In particular, for any integer \( g \) such that \( 0 \leq g \leq f \leq e \), we have \( (\phi_f)_g = \phi_g \).
\end{corollary}
\begin{proof}
It follows directly from the definition of \( \mathcal{A}_{p^e} \) that
\[
\phi_f \in \mathcal{A}_{p^f} \quad \text{and} \quad (\phi_f)_g = \phi_g.
\]
Thus, the map is well-defined.
We now prove that $\Theta$ is a group homomorphism.

Let \( \phi, \phi' \in \mathcal{A}_{p^e} \).
We need to show that
\[
(\phi \circ \phi')_f = \phi_f \circ \phi_f'.
\]
Let \( r \in \mathbb{Z} \) with \( \gcd(r, p) = 1 \).
By the definition of \( \phi_f' \), we have \( \phi_f'([r]_{p^f}) = [s]_{p^f} \), where \( s \) is the integer such that \( \phi'([r]) = [s] \).
Similarly, \( \phi_f([s]_{p^f}) = [s']_{p^f} \), where \( s' \) is the integer such that \( \phi([s]) = [s'] \).
Therefore, we have
\[
\phi([s]) = \phi(\phi'([r])) = [s'].
\]
By the definition of \( (\phi \circ \phi')_f \), \( \phi_f \) and \( \phi_f' \), we deduce that
\[
(\phi \circ \phi')_f([r]_{p^f}) = [s']_{p^f}= \phi_f \circ \phi_f'([r]_{p^f}).
\]
which proves that the map is a group homomorphism.w
\end{proof}

We now compute the monoid \( \mathcal{A}_{p^e} \).
\begin{proposition}\label{Prop2.7}
We have:
\[
\mathcal{A}_{p^e} =
\begin{cases}
    \Aut(U_{p^e}, \cdot), & \text{if } p \text{ is odd}, \\
    \{ \Id \}, & \text{if } p = 2 \text{ and } e \leq 2, \\
    \Aut_{\langle[5]\rangle}(U_{2^e}, \cdot) = \{ \tilde{\chi}_{8,t} \mid t \in \{0,1\} \}, & \text{if } p = 2 \text{ and } e = 3, \\
    \Aut_{\langle [5] \rangle}(U_{2^e}, \cdot) = \{ \chi_{2^e,(t_1, 0, t_3)} \mid (t_1, 0, t_3) \in T_{2^e} \}, & \text{if } p = 2 \text{ and } e \geq 4.
\end{cases}
\]
Here, \( \Aut_{\langle [5] \rangle}(U_{2^e}, \cdot) \) denotes the subgroup of \( \Aut(U_{2^e}, \cdot) \) consisting of automorphisms that preserve \( \langle [5] \rangle \) setwise.
\( T_{2^e} \) and \( \chi_{2^e,(t_1, 0, t_3)} \) are defined in Proposition \ref{Prop2.3}, and for any \( t \in \{0,1\} \), \( \tilde{\chi}_{8,t} \) is the automorphism of \( U_8 \) that maps \( [-1] \) to \( [-1][5]^t \) and \( [5] \) to \( [5] \).
\end{proposition}
\begin{proof}
Let \( \phi \in \operatorname{Aut}(U_{p^e}, \cdot) \).
By Definition-Lemma \ref{Deflem2.5}, we know that \( \phi \in \mathcal{A}_{p^e} \) if and only if the following condition holds:
for all \( r, r', s, s' \in \mathbb{Z} \) with \( \gcd(r, p) = \gcd(r', p) = 1 \), if \( \phi([r]) = [s] \) and \( \phi([r']) = [s'] \), then for every integer \( f \) such that \( 0 \leq f \leq e \), we have
\begin{equation}\label{iff}
    [r]_{p^f} = [r']_{p^f} \quad \text{if and only if} \quad [s]_{p^f} = [s']_{p^f}.
\end{equation}

\begin{description}
    \item[\sf{When $p$ is odd}]  Let \( r, r', s, s' \in \mathbb{Z} \) with \( \gcd(r, p) = \gcd(r', p) = 1 \), and let \( \phi([r]) = [s] \), \( \phi([r']) = [s'] \), with \( a \) being a generator of \( U_{p^e} \).
    By Proposition \ref{Prop2.3}, there exists \( t \in \mathbb{Z} \) such that \( \gcd(t, p^{e-1}(p-1)) = 1 \) and \( \phi(a) = a^t \).
    We can express \( [r] = a^u \) and \( [r'] = a^{u'} \) for some \( u, u' \in \mathbb{Z} \).

    Proving that Equation (\ref{iff}) holds is equivalent to showing that
    \[
    a^u = a^{u'} \quad \text{if and only if} \quad a^{ut} = a^{u't}.
    \]
    The latter follows from the equivalence of the following statements:
    \begin{enumerate}[label=(\roman*)]
        \item \( a^u = a^{u'} \);
        \item \( [u]_{p^{f-1}(p-1)} = [u']_{p^{f-1}(p-1)} \);
        \item \( [ut]_{p^{f-1}(p-1)} = [u't]_{p^{f-1}(p-1)} \);
        \item \( a^{ut} = a^{u't} \).
    \end{enumerate}
    Note that (ii) is equivalent to (iii) because \( \gcd(t, p^{e-1}(p-1)) = 1 \).
    Thus, the Proposition is proven for \( p \) odd.

    \item[\sf{When $p=2$ and \(e\leq 2\)}] we have $\phi = \Id$, and the result follows immediately.

    \item[\sf{When $p=2$ and \(e=3\)}] By Proposition \ref{Prop2.3} (2)(b), the automorphism \( \phi \) is given by:
    \[
    \phi([-1]) = [-1]^{t_3} [5]^{t_4}, \quad \phi([5]) = [-1]^{t_5} [5]^{t_6},
    \]
    where \( (t_3, t_4, t_5, t_6) \) belongs to the set
    \[
    \{ (1,0,0,1), (1,0,1,1), (0,1,1,0), (0,1,1,1), (1,1,1,0), (1,1,0,1) \}.
    \]
    Since \( [5]_4 = [1]_4 \), the equivalence (\ref{iff}) is satisfied for \( r = 5 \), \( r' = 1 \), and \( f = 2 \) if and only if \( [-1]_4^{t_5} = [1]_4 \).
    This implies \( t_5 = 0 \), and consequently, \( t_3 = t_6 = 1 \), while \( t_4 \) remains either \( 0 \) or \( 1 \).

    Now, setting \( t_5 = 0 \), let \( r, r', s, s' \in \mathbb{Z} \) be such that \( \gcd(r, p) = \gcd(r', p) = 1 \) and
    \[
    \phi([r]) = [s] \quad\text{and}\quad \phi([r']) = [s'].
    \]
    Expressing \( [r] \) and \( [r'] \) in terms of the generators, we write
    \[
    [r] = [-1]^v [5]^w \quad\text{and}\quad [r'] = [-1]^{v'} [5]^{w'},
    \]
    for some \( v, v' \in \{0,1\} \) and \( w, w' \in \llbracket 0, 2^{e-2} - 1 \rrbracket \).
    Then, applying \( \phi \), we obtain:
    \[
    \phi([r]) = [-1]^v [5]^{t_4 v + w} \quad\text{and}\quad \phi([r']) = [-1]^{v'} [5]^{t_4 v' + w'}.
    \]
    Thus the equivalence (\ref{iff}) holds automatically for \( f \leq 3 \).
    Thus, the corollary follows.

    \item[\sf{When $p=2$ and \(e\geq 4\)}]
    By Proposition \ref{Prop2.3} (c), there exists \( (t_1, t_2, t_3) \in T_{2^e} \) such that
    \[
    \phi([-1]) = [-1][5]^{t_1} \quad\text{and}\quad \phi([5]) = [-1]^{t_2} [5]^{t_3}.
    \]
    Since \( [5]_4 = [1]_4 \), the equivalence (\ref{iff}) is satisfied for \( r = 5 \), \( r' = 1 \), and \( f = 2 \) if and only if \( [-1]_4^{t_2} = [1]_4 \), which implies \( t_2 = 0 \).

    We now set \( t_2 = 0 \).
    Let \( r, r', s, s' \in \mathbb{Z} \) be such that \( \gcd(r, p) = \gcd(r', p) = 1 \) and
    \[
    \phi([r]) = [s] \quad\text{and}\quad \phi([r']) = [s'].
    \]
    Expressing \( [r] \) and \( [r'] \) in terms of generators, we write
    \[
    [r] = [-1]^v [5]^w \quad\text{and}\quad [r'] = [-1]^{v'} [5]^{w'},
    \]
    for some \( v, v' \in \{0, 1\} \) and \( w, w' \in \llbracket 0, 2^{e-2} - 1 \rrbracket \).
    Applying \( \phi \), we obtain:
    \[
    \phi([r]) = [-1]^v [5]^{t_1 v + t_3 w} \quad\text{and}\quad \phi([r']) = [-1]^{v'} [5]^{t_1 v' + t_3 w'}.
    \]
    The equivalence (\ref{iff}) holds automatically for \( f \leq 2 \).
    Now, suppose \( f \geq 3 \).
    The following statements are equivalent:
    \begin{enumerate}[label=(\roman*)]
        \item \( [(-1)^{v} 5^{t_1 v + t_3 w}]_{2^f} = [(-1)^{v'} 5^{t_1 v' + t_3 w'}]_{2^f} \).
        \item \( [v - v']_2 = [0]_2 \) and \( [t_1 (v - v') + t_3 (w - w')]_{2^{f-2}} = [0]_{2^{f-2}} \).
        \item \( [v - v']_2 = [0]_2 \) and \( [w - w']_{2^{f-2}} = [0]_{2^{f-2}} \).
        \item \( [(-1)^v 5^w]_{2^f} = [(-1)^{v'} 5^{w'}]_{2^f} \).
    \end{enumerate}
    Indeed,
    \begin{itemize}
        \item The equivalences \( (i) \iff (ii) \) and \( (iii) \iff (iv) \) follow directly from the fact that \( U_{2^e} \) is the direct product of the cyclic group generated by \( [5] \), which has order \( 2^{e-2} \), and the subgroup generated by \( [-1] \), which has order 2.
        \item Since \( f \geq 3 \) and \( t_1 \in \{0, 2^{e-3}\} \), it is clear that \( (iii) \implies (ii) \).
        \item To prove \( (ii) \implies (iii) \), given \( t_1 \in \{0, 2^{e-3}\} \) and \( \gcd(t_3, 2) = 1 \), we deduce that \( [t_1 (v - v') + t_3 (w - w')]_{2^{f-2}} = [0]_{2^{f-2}} \) implies \( [w - w']_{2^{f-2}} = [0]_{2^{f-2}} \).
    \end{itemize}
    Thus, the corollary follows. \qedhere
\end{description}
\end{proof}

\begin{remark}\label{Ape}
We observe that the map \( \Theta \), defined in Corollary \ref{theta}, is surjective in the following cases:
\begin{itemize}
    \item when \( p \) is odd;
    \item when \( p = 2 \) and \( e \leq 3 \) or \( f \leq 2 \).
\end{itemize}
However, for \( p = 2 \) with \( e \geq 4 \), \( f \geq 4 \), and \( f < e \), the element \( \chi_{2^f, (2^{f-3}, 0, t_3)} \) has no preimage for any odd \( t_3 \), using the notation from Proposition \ref{Prop2.3}, and \( \Theta( \chi_{2^e, (2^{e-3}, 0, t_3)}) =  \chi_{2^f, (0, 0, t_3)} \).
Moreover, when \( p = 2 \), \( e \geq 4 \), and \( f = 3 \), Proposition \ref{Prop2.7} gives \( \Theta(\phi) = \Id \) for all \( \phi \in \mathcal{A}_{p^e} \).
\end{remark}

From Proposition \ref{Prop2.3}, we obtain the following group isomorphisms:
\begin{itemize}
    \item If \( p \) is odd, then \( (\Aut(U_{p^e}, \cdot), \circ) \cong (U_{p^{e-1}(p-1)}, \cdot) \).
    \item If \( e \leq 2 \), then \( (\Aut(U_{2^e}, \cdot), \circ) = (\{\Id\}, \circ) \).
    \item If \( e = 3 \), then \( (\Aut(U_{2^3}, \cdot), \circ) \cong (S_3, \circ) \).
\end{itemize}

The remainder of this section is dedicated to characterizing the group structure of \( (\Aut(U_{2^e}, \cdot), \circ) \) for any \( e \geq 4 \).
To establish the necessary notation, we recall the definition of a central product (see also \cite[Theorem 5.3]{zbMATH03724683}).
\begin{definition}\label{Def5.9}
Let \( G_1 \) and \( G_2 \) be groups.
\begin{enumerate}
    \item Given a group \( C \) and a pair of injective group homomorphisms \(\iota = (\iota_1, \iota_2)\), where each \(\iota_i : C \hookrightarrow Z(G_i)\) for \(i \in \{1, 2\}\), we define the \textsf{external central product with respect to \((G_1, G_2, C, \iota)\)}, denoted by \( G_1 \centralprod_{\iota} G_2 \), as the quotient of the direct product \( G_1 \times G_2 \) endowed with a binary operation defined componentwise by its diagonal subgroup \(\Delta(C) = \{ (\iota_1(c), \iota_2(c)) \mid c \in C \}\), that is
    \[
    G_1 \centralprod_{\iota} G_2 = \frac{G_1 \times G_2}{\Delta(C)}.
    \]

    \item We say that a group \( G \) is a \textsf{central product of \( G_1 \) and \( G_2 \)} if there exists a group \( C \) and \(\iota = (\iota_1, \iota_2)\), where \(\iota_i : C \hookrightarrow Z(G_i)\) is an injective group morphism for \( i \in \{1, 2\} \), such that \( G = G_1 \centralprod_{\iota} G_2 \).

    \item We say that \( G \) is the \textsf{internal central product of \( G_1 \) and \( G_2 \)} if \( G_i \leq G \) for \( i \in \{1, 2\} \), \( G = G_1 G_2 \), and the elements of $G_1$ and $G_2$ commute.
\end{enumerate}
\end{definition}

We first observe that if \( C = \{1\} \) and \(\iota = (\iota_1, \iota_2)\), where each \(\iota_i : \{1\} \hookrightarrow Z(G_i)\) maps \(1\) to \(1\) for \( i \in \{1, 2\} \), then the central product \( G_1 \centralprod_{\iota} G_2 \) reduces to the direct product \( G_1 \times G_2 \).
Moreover, the concepts of internal and external central products coincide.
Specifically, any external central product \( G_1 \centralprod_{\iota} G_2 \) can be realized as an internal central product of the subgroups
\[
\frac{G_1 \times \iota_2(C)}{\Delta (C)} \quad \text{and} \quad \frac{\iota_1(C) \times G_2}{\Delta (C)}.
\]
Conversely, if \( G \) is an internal central product of \( G_1 \) and \( G_2 \), then defining \( C = G_1 \cap G_2 \) with the natural inclusion maps \(\iota_i : C \to Z(G_i)\) yields an isomorphism \( G \cong G_1 \centralprod_{\iota} G_2 \).

With this framework in place, we now define a group \((V_e, \star)\) explicitly and establish its isomorphism with \( (\Aut(U_{2^e}, \cdot), \circ) \).
This characterization allows us to determine the precise structure of the automorphism group.
\begin{definitionlemma} \label{def:ve}
Let $e \geq 4$, \( V_e = U_{2^{e-2}} \times \mathbb{Z}/2\mathbb{Z} \times \mathbb{Z}/2\mathbb{Z} \), and \(\star\) be the binary operation on \(V_e\) defined for all \(a_1, a_2, a_3, a_1', a_2', a_3' \in \mathbb{Z}\) with $\gcd(a_1, 2)= \gcd(a_1',2)=1$ as:
\[ \left([a_1]_{2^{e-2}}, [a_2]_2, [{a_3}]_2\right) \star \left([a_1']_{2^{e-2}}, [a_2']_2,[a_3']_2\right) = \left([a_1 a_1' + 2^{e-3} a_3 a_2']_{2^{e-2}}, [a_2 + a_2']_2, \left[a_3 + a_3'\right]_2\right). \]

Then \( (V_e, \star) \) is a group.
In particular:
\begin{itemize}
    \item its identity element is \(([1]_{2^{e-2}}, [0]_2, [0]_2)\), and
    \item the inverse of $\left([a_1]_{2^{e-2}}, [a_2]_2, [{a_3}]_2\right)$ with respect to \(\star\) is:
    \[ \left([1 + 2^{e-3} a_3 a_2]_{2^{e-2}} [a_1]_{2^{e-2}}^{-1}, [-a_2]_2, [{-a_3}]_2 \right). \]
\end{itemize}
\end{definitionlemma}

\begin{proof}
\begin{description}
    \item[\sf{Associativity of $\star$}]
    \begin{align*}
        &\mathrel{\phantom{=}} \left( \left([a_1]_{2^{e-2}}, [a_2]_2, [{a_3}]_2\right) \star \left([b_1]_{2^{e-2}}, [b_2]_2, [{b_3}]_2\right) \right) \star \left([c_1]_{2^{e-2}}, [c_2]_2, [{c_3}]_2\right) \\
        &= \left([a_1b_1 +2^{e-3} a_3b_2]_{2^{e-2}}, [a_2+b_2]_2, [a_3+b_3]_2\right) \star \left([c_1]_{2^{e-2}}, [c_2]_2, [{c_3}]_2\right) \\
        &= \left([a_1b_1c_1 +2^{e-3}(a_3b_2c_1+a_3c_2 +b_3c_2)]_{2^{e-2}}, [a_2+b_2+c_2]_2, [a_3+b_3+c_3]_2\right) \\
        &= \left([a_1b_1c_1 +2^{e-3}(a_3b_2+a_3c_2 +a_1b_3c_2)]_{2^{e-2}}, [a_2+b_2+c_2]_2, [a_3+b_3+c_3]_2\right) \\
        &= \left([a_1]_{2^{e-2}}, [a_2]_2, [{a_3}]_2\right) \star \left([b_1c_1+2^{e-3}b_3c_2]_{2^{e-2}}, [b_2+c_2]_2, [b_3+c_3]_2\right) \\
        &= \left([a_1]_{2^{e-2}}, [a_2]_2, [{a_3}]_2\right) \star \left( \left([b_1]_{2^{e-2}}, [b_2]_2, [{b_3}]_2\right) \star \left([c_1]_{2^{e-2}}, [c_2]_2, [{c_3}]_2\right) \right).
    \end{align*}
The third equality holds due to the conditions \(\gcd(c_1, 2) = 1\) and \(\gcd(a_1, 2) = 1\).

    \item[\sf{Identity element}]
    \begin{align*}
        &([1]_{2^{e-2}}, [0]_2, [0]_2) \star \left([a_1]_{2^{e-2}}, [a_2]_2, [{a_3}]_2\right) \\
        &= \left([1\cdot a_1 +2^{e-3}\cdot 0\cdot a_2]_{2^{e-2}}, [0+a_2]_2, [0+a_3]_2\right) \\
        &= \left([a_1]_{2^{e-2}}, [a_2]_2, [{a_3}]_2\right) \\
        &= \left([a_1\cdot 1 +2^{e-3}\cdot a_3\cdot 0]_{2^{e-2}}, [a_2+0]_2, [a_3+0]_2\right) \\
        &= \left([a_1]_{2^{e-2}}, [a_2]_2, [{a_3}]_2\right) \star ([1]_{2^{e-2}}, [0]_2, [0]_2).
    \end{align*}

    \item[\sf{Inverse element}]
    \begin{align*}
        &\mathrel{\phantom{=}} \left([1 + 2^{e-3} a_3 a_2]_{2^{e-2}} [a_1]_{2^{e-2}}^{-1}, [-a_2]_2, [{-a_3}]_2 \right) \star \left([a_1]_{2^{e-2}}, [a_2]_2, [{a_3}]_2\right) \\
        &= \left([1 +2^{e-3}a_3a_2 +2^{e-3}(-a_3)a_2]_{2^{e-2}}, [-a_2+a_2]_2, [-a_3+a_3]_2\right) \\
        &= \left([1]_{2^{e-2}}, [0]_2, [0]_2\right) \\
        \shortintertext{and}
        &\mathrel{\phantom{=}} \left([a_1]_{2^{e-2}}, [a_2]_2, [{a_3}]_2\right) \star  \left([1 + 2^{e-3} a_3 a_2]_{2^{e-2}} [a_1]_{2^{e-2}}^{-1}, [-a_2]_2, [{-a_3}]_2 \right)\\
        &= \left([1 +2^{e-3}a_3a_2 -2^{e-3}a_3a_2]_{2^{e-2}}, [-a_2+a_2]_2, [-a_3+a_3]_2\right) \\
        &= \left([1]_{2^{e-2}}, [0]_2, [0]_2\right). \qedhere
    \end{align*}
\end{description}
\end{proof}
We note that the condition \( e \geq 4 \) is necessary for the above definition to be well-defined.
If \( e < 4 \), then \( 2^{e-3} \) is not even, and consequently, the term \( a_1a_1' + 2^{e-3}a_3a_2' \) in the definition of \( \star \) may fail to be odd, violating the required conditions to be in $U_{2^{e-2}}$.

The following lemma establishes key properties of the group \( (V_e, \star) \).
\begin{lemma} \label{vep}
For $e\geq 4$, in the group \( (V_e, \star) \), for all \( k \in \mathbb{Z} \) and \( a_1 \in \mathbb{Z} \) with \( \gcd(a_1,2) = 1 \), the following properties hold:

\begin{enumerate}
\item
    \begin{itemize}
        \item \( ([a_1]_{2^{e-2}}, [0]_2, [0]_2)^k = ([a_1^k]_{2^{e-2}}, [0]_2, [0]_2) \),
        \item \( ([a_1]_{2^{e-2}}, [0]_2, [1]_2)^k = ([a_1^k]_{2^{e-2}}, [0]_2, [k]_2) \),
        \item \( ([a_1]_{2^{e-2}}, [1]_2, [0]_2)^k = ([a_1^k]_{2^{e-2}}, [k]_2, [0]_2) \).
    \end{itemize}
\item ~\vspace{-\baselineskip}
    \[
    ([a_1]_{2^{e-2}}, [1]_2, [1]_2)^k =
    \begin{cases}
        ([a_1^k]_{2^{e-2}}, [0]_2, [0]_2), & \text{if } [k]_4 =[0]_4, \\
        ([a_1^k]_{2^{e-2}}, [1]_2, [1]_2), & \text{if } [k]_4 =[1]_4, \\
        ([a_1^k + 2^{e-3}]_{2^{e-2}}, [0]_2, [0]_2), & \text{if } [k]_4 =[2]_4, \\
        ([a_1^k + 2^{e-3}]_{2^{e-2}}, [1]_2, [1]_2), & \text{if } [k]_4 =[3]_4.
    \end{cases}
    \]
\item The center of \( V_e \) is given by
    \[
    Z(V_e) = \left\{ \left([a_1]_{2^{e-2}}, [0]_2, [0]_2\right) \middle\mid a_1 \in \mathbb{Z}, \ \gcd(a_1,2) = 1 \right\}.
    \]
    In particular, \( |Z(V_e)| = 2^{e-3} \).
\item The subgroup \( H_1 \) of \( V_e \), generated by \( ([1]_{2^{e-2}}, [1]_2, [1]_2) \) and \( ([1]_{2^{e-2}}, [1]_2, [0]_2) \), is isomorphic to \( D_4 \).
    Moreover
    \begin{enumerate}
        \item when $e=4$, $V_e=H$;
        \item when $e\geq 5$, $H_1 = \left\{ \left([5^{2^{e-5}w}]_{2^{e-2}}, [a_2]_2, [a_3]_2 \right) \middle\mid w \in \{0,1\}, \ a_2, a_3 \in \{0,1\} \right\}$.
    \end{enumerate}
\end{enumerate}
\end{lemma}

\begin{proof}
\begin{enumerate}
\item[(1), (2)]  These are easily proven using induction on \(k\).
\item[(3)] Let \( a_1, a_1' \in \mathbb{Z} \) with \( \gcd(a_1,2) = \gcd(a_1',2) = 1 \).
It is clear that
\[
\left([a_1]_{2^{e-2}}, [0]_2, [0]_2\right) \in Z(V_e).
\]
For the reverse inclusion, we analyze the following calculations:
\begin{align}
    \label{Eq03}
    ([a_1]_{2^{e-2}}, [0]_2, [1]_2) \star ([a_1']_{2^{e-2}}, [1]_2, [0]_2) &= ([a_1 a_1' + 2^{e-3}]_{2^{e-2}}, [1]_2, [1]_2), \\
    \label{Eq04}
    ([a_1']_{2^{e-2}}, [1]_2, [0]_2) \star ([a_1]_{2^{e-2}}, [0]_2, [1]_2) &= ([a_1 a_1']_{2^{e-2}}, [1]_2, [1]_2), \\
    \label{Eq05}
    ([a_1]_{2^{e-2}}, [1]_2, [1]_2) \star ([a_1']_{2^{e-2}}, [1]_2, [0]_2) &= ([a_1 a_1' + 2^{e-3}]_{2^{e-2}}, [0]_2, [1]_2), \\
    \label{Eq06}
    ([a_1']_{2^{e-2}}, [1]_2, [0]_2) \star ([a_1]_{2^{e-2}}, [1]_2, [1]_2) &= ([a_1 a_1']_{2^{e-2}}, [0]_2, [1]_2).
\end{align}
Since $[a_1a_1']_{2^{e-2}} \neq [a_1a_1' +2^{e-3}]_{2^{e-2}}$, comparing Equations (4) and (5) give that the elements \( ([a_1]_{2^{e-2}}, [0]_2, [1]_2) \), and \( ([a_1]_{2^{e-2}}, [1]_2, [0]_2) \) cannot belong to $Z(V_e)$.
Similarly, comparing Equations (6) and (7) give that \( ([a_1]_{2^{e-2}}, [1]_2, [1]_2) \) cannot belong to \( Z(V_e) \), completing the proof.
\item[(4)] We define \( v = ([1]_4, [1]_2, [1]_2) \) and \( u = ([1]_{2^{e-2}}, [1]_2, [0]_2) \).
By (2), \( v \) has order \( 4 \) and \( u \) has order \( 2 \).
Moreover, we have the relation
\[
v u v = u.
\]
This confirms that the subgroup
\[
\langle ([1]_{2^{e-2}}, [1]_2, [1]_2), ([1]_{2^{e-2}}, [1]_2, [0]_2) \rangle
\]
of \( (V_e, \star) \) is isomorphic to \( D_4 \).
In particular, we have
\begin{align*}
    &\mathrel{\phantom{=}} \langle ([1]_{2^{e-2}}, [1]_2,[1]_2),([1]_{2^{e-2}},[1]_2,[0]_2)\rangle \\
    &= \left\lbrace ([1]_{2^{e-2}}, [1]_2,[1]_2)^{i}\star([1]_{2^{e-2}},[1]_2,[0]_2)^{j}\middle| i\in \lbrace 0,1,2,3\rbrace, j\in\lbrace 0,1\rbrace \right\rbrace.
\end{align*}
Furthermore, we have:
\begin{align*}
    &([1]_{2^{e-2}}, [1]_2, [1]_2)^i \star ([1]_{2^{e-2}}, [1]_2, [0]_2)^j \\
    &=
    \begin{cases}
    ([1]_{2^{e-2}}, [j]_2, [0]_2), & \text{if } i=0, \\
    ([1+2^{e-3}j]_{2^{e-2}}, [1+j]_2, [1]_2), & \text{if } i=1, \\
    ([1+2^{e-3}]_{2^{e-2}}, [j]_2, [0]_2), & \text{if } i=2, \\
    ([1+2^{e-3}(j+1)]_{2^{e-2}}, [1+j]_2, [1]_2), & \text{if } i=3.
    \end{cases}
\end{align*}

\begin{enumerate}
    \item When $e=4$, from the computations above, we have $V_e=H$.
    \item When $e\geq 5$, from the computation above and Lemma \ref{Lem5.14}, we deduce that
    \[
    H_1 = \left\{ \left([5^{2^{e-5}w}]_{2^{e-2}}, [a_2]_2, [a_3]_2 \right) \mid w \in \{0,1\}, \ a_2, a_3 \in \{0,1\} \right\}. \qedhere
    \]
\end{enumerate}
\end{enumerate}
\end{proof}

The next theorem describes the structure of \(\Aut(U_{2^e}, \cdot)\) for \( e \geq 4 \).
\begin{theorem}\label{Theo2.6}
For \( e \geq 4 \), consider the map
\[
    \begin{array}{cccc}
    \Phi: & (\Aut(U_{2^e}, \cdot), \circ) & \to & (V_e, \star) \\
    & \chi_{2^e, (t_1, t_2, t_3)} & \mapsto & \left([t_3]_{2^{e-2}}, [t_2]_2, \left[t_1/2^{e-3}\right]_2\right),
    \end{array}
\]
where $(t_1, t_2, t_3)\in T_{2^e}$ with
        \[
           \ \ \ \ \ \ T_{2^e} = \left\{ (t_1, t_2, t_3) \middle| t_1 \in \{0, 2^{e-3}\}, t_2 \in \{0, 1\}, t_3 \in \llbracket 0, 2^{e-2}-1 \rrbracket, \gcd(t_3, 2) = 1 \right\},
        \]
(see also Proposition \ref{Prop2.3}).

Then, we have the following results:
\begin{enumerate}
    \item \(\Phi\) is well-defined and a group isomorphism.

    \item When \( e = 4 \), \( (\Aut(U_{2^4}, \cdot), \circ) \) is isomorphic to the dihedral group of order 8.
    More precisely,
    \[
    V_e \cong \left\langle r, s \middle| r s r = s \wedge s^2 = 1 \wedge r^4 = 1 \right\rangle,
    \]
    where the isomorphism sends \( ([1]_4, [1]_2, [1]_2) \) to \( r \) and \( ([1]_4, [0]_2, [1]_2) \) to \( s \).

    \item When \( e \geq 5 \), we have the isomorphism
\[
(\Aut(U_{2^e}, \cdot),\circ) \cong \mathbb{Z}/2\mathbb{Z} \times \left( D_4 \centralprod_{\iota} \mathbb{Z}/2^{e-4}\mathbb{Z} \right),
\]
where the central product is taken with respect to \( \iota = (\iota_1, \iota_2) \) where \( \iota_1 \) and \( \iota_2 \) are the unique injective morphisms from \( \mathbb{Z}/2\mathbb{Z} \) into \( Z(D_4) \) and \( Z(\mathbb{Z}/2^{e-4}\mathbb{Z}) \), respectively.

More explicitly, the group \( V_e \) decomposes as the internal direct product
\[
V_e = \langle ([-1]_{2^{e-2}}, [0]_2, [0]_2) \rangle \times H,
\]
where \( H \) is given by
\[
H = \left\{ \left([5]_{2^{e-2}}^w, [a_2]_2, [a_3]_2 \right) \middle| w \in \llbracket 0, 2^{e-4}-1 \rrbracket, \ a_2, a_3 \in \{0,1\} \right\}.
\]
Furthermore, \( H \) itself is the internal central product of the subgroups
\begin{align*}
    H_1 &= \left\{ \left([5^{2^{e-5}w}]_{2^{e-2}}, [a_2]_2, [a_3]_2 \right) \middle| w \in \{0,1\}, \ a_2, a_3 \in \{0,1\} \right\} \\
    \shortintertext{and}
    H_2 &= \left\{ ([5]_{2^{e-2}}^w, [0]_2, [0]_2) \middle| w \in \llbracket 0, 2^{e-4}-1 \rrbracket \right\}.
\end{align*}
\end{enumerate}
In particular, we have \( \left|\Aut(U_{2^e}, \cdot)\right| = 2^{e-1} \).
\end{theorem}

\begin{proof} ~
\begin{enumerate}
\item Let \( t = (t_1, t_2, t_3) \) and \( t' = (t_1', t_2', t_3') \) be elements of \( T_{2^e} \).
We begin by establishing the well-definedness and injectivity of \( \Phi \).
That is, we need to prove that \( \chi_{2^e, t} = \chi_{2^e, t'} \) if and only if
\[
([t_3]_{2^{e-2}}, [t_2]_2, [t_1 / 2^{e-3}]_2) = ([t_3']_{2^{e-2}}, [t_2']_2, [t_1' / 2^{e-3}]_2).
\]
Since \( \{ [-1], [5] \} \) forms a system of generators for \( U_{2^e} \), \( \chi_{2^e, t} \) and \( \chi_{2^e, t'} \) are equal if and only if they agree on both \( [-1] \) and \( [5] \).
Specifically, the equations
\[
\chi_{2^e, t} ([-1]) = \chi_{2^e, t'} ([-1]) \quad \text{and} \quad \chi_{2^e, t} ([5]) = \chi_{2^e, t'} ([5])
\]
are equivalent to
\[
[-1][5]^{t_1} = [-1][5]^{t_1'} \quad \text{and} \quad [-1]^{t_2}[5]^{t_3} = [-1]^{t_2'}[5]^{t_3'},
\]
which in turn are equivalent to the following congruences:
\[
[t_1/2^{e-3}]_2 = [t_1'/2^{e-3}]_2, \quad [t_2]_2 = [t_2']_2, \quad \text{and} \quad [t_3]_{2^{e-2}} = [t_3']_{2^{e-2}},
\]
since \( t_1, t_1' \in \{ 0, 2^{e-3} \} \).

The surjectivity of \( \Phi \) is straightforward to verify.
To show that \( \Phi \) is a group morphism, consider \( \phi_t \) and \( \phi_{t'} \) in \( \Aut(U_{2^e}, \cdot) \).
Since $t_1,t_1' \in \{0,2^{e-3}\}$, $t_1$ is even and $[5^{2t_1'}] = [1]$.
Thus since $t_3$ is odd, we deduce that \( [5^{t_3 t_1'}] = [5^{t_1'}] \).
Consequently, we have that:
\begin{align*}
    \chi_{2^e, t} \circ \chi_{2^e, t'} ([-1])
    &= \chi_{2^e, t} ([-1][5]^{t_1'}) \\
    &= [-1]^{1 + t_2 t_1'}[5]^{t_1 + t_3 t_1'} \\
    &= [-1]^1[5]^{t_1 + t_1'}.
\end{align*}

Next, since \( \gcd(t_3', 2) = 1 \) and \( t_1 \in \{ 0, 2^{e-3} \} \), we obtain:
\begin{align*}
    \chi_{2^e, t} \circ \chi_{2^e, t'} ([5])
    &= \chi_{2^e, t} ([-1]^{t_2'}[5]^{t_3'}) \\
    &= [-1]^{t_2' + t_2 t_3'} [5]^{t_1 t_2' + t_3 t_3'} \\
    &= [-1]^{t_2+t_2'} [5]^{t_1 t_2' + t_3 t_3'}.
\end{align*}
Thus, we obtain:
\[
\Phi(\chi_{2^e, t} \circ \chi_{2^e, t'}) = ([t_3 t_3' + t_2' t_1]_{2^{e-2}}, [t_2 + t_2']_2, [(t_1 + t_1') / 2^{e-3}]_2).
\]
and
\[ \Phi(\chi_{2^e, t}) \star \Phi(\chi_{2^e, t'}) = ([t_3]_{2^{e-2}}, [t_2]_2, [t_1 / 2^{e-3}]_2) \star ([t_3']_{2^{e-2}}, [t_2']_2, [t_1' / 2^{e-3}]_2). \]
Thus, by the definition of \( \star \) (see Definition-Lemma \ref{def:ve}), we conclude that:
\[
\Phi(\chi_{2^e, t} \circ \chi_{2^e, t'}) = \Phi(\chi_{2^e, t}) \star \Phi(\chi_{2^e, t'}),
\]
confirming that \( \Phi \) is a group homomorphism.

\item Follows from Lemma \ref{vep} (4). 

\item When $e\geq 5$, we prove that \( V_e \) is the internal direct product of \( \langle ([-1]_{2^{e-2}}, [0]_2, [0]_2) \rangle \) and $H$.
We proceed with the following steps:
\begin{itemize}
    \item \textsf{Subgroup Property of \( H \):} We begin by proving that \( H \) is a subgroup of \( V_e \).
    The element \( ([1]_{2^{e-2}}, [0]_2, [0]_2) \) is clearly in \( H \).
    Let \( a = ([5^w]_{2^{e-2}}, [a_2]_2, [a_3]_2) \) and \( a' = ([5^{w'}]_{2^{e-2}}, [a_2']_2, [a_3']_2) \) be elements of \( H \).
    We have:
    \[
    a \star a' = ([5^{w+w'} + 2^{e-3}a_3 a_2']_{2^{e-2}}, [a_2 + a_2']_2, [a_3 + a_3']_2),
    \]
    and the inverse of \( a \) with respect to $\star$ is given by
    \[
    a^{-1} = ([1 + 2^{e-3}a_2 a_3]_{2^{e-2}} [5]_{2^{e-2}}^{-w}, [-a_2]_2, [-a_3]_2).
    \]
    Clearly, when either \( [a_3]_2 \) or \( [a_2']_2 \) is zero, these expressions remain in \( H\).
    Otherwise, the closure under $\star$ and inverses for \( H\) follows from Lemma \ref{Lem5.14}.

    \item \textsf{Normality of \( H \):} To prove that \( H \) is normal in \( V_e \), let \( a = ([5^w]_{2^{e-2}}, [a_2]_2, [a_3]_2) \in H \) and \( b = ([b_1]_{2^{e-2}}, [b_2]_2, [b_3]_2) \in V_e \).
    We have:
    \[
    b \star a \star b^{-1} = ([5^w + 2^{e-3}(b_3 a_2 - a_3 b_2)]_{2^{e-2}}, [a_2]_2, [a_3]_2).
    \]
    This result is clearly in \( H \), as it holds when \( [b_3 a_2 - a_3 b_2]_2 \) is zero, and otherwise by Lemma \ref{Lem5.14}.
    Therefore, \( H \) is a normal subgroup of \( V_e \).

    \item \textsf{Order of \( \langle ([-1]_{2^{e-2}}, [0]_2, [0]_2) \rangle \):} The group \( \langle ([-1]_{2^{e-2}}, [0]_2, [0]_2) \rangle \) has order 2 and is normal in \( V_e \) (see Lemma \ref{vep} (1) and (3)).

    \item \textsf{Intersection:} We have that:
    \[
    H \cap \langle ([-1]_{2^{e-2}}, [0]_2, [0]_2) \rangle = \{ ([1]_{2^{e-2}}, [0]_2, [0]_2) \},
    \]
    since \( U_{2^{e-2}} \) is the direct product of the subgroups generated by \( [5]_{2^{e-2}} \) and \( [-1]_{2^{e-2}} \).

    \item \textsf{Internal Direct Product:} Let \( ([a_1]_{2^{e-2}}, [a_2]_2, [a_3]_2) \in V_e \).
    We can express \( [a_1]_{2^{e-2}} \) as \( [-1]^v_{2^{e-2}} [5]^w_{2^{e-2}} \), where \( v \in \{0,1\} \) and \( w \in \llbracket 0, 2^{e-4}-1\rrbracket \).
    Thus, we can write
    \[
    ([a_1]_{2^{e-2}}, [a_2]_2, [a_3]_2) = ([5]^w_{2^{e-2}}, [a_2]_2, [a_3]_2) \star ([-1]^v_{2^{e-2}}, [0]_2, [0]_2),
    \]
    where \( \star \) denotes the binary operation in \( V_e \).
    This concludes the proof that \( V_e \) is the internal direct product of \( H \) and \( \langle ([-1]_{2^{e-2}}, [0]_2, [0]_2) \rangle \).

    \item \textsf{Internal Central Product:} To show that \( H \) is the internal central product of \( H_1 \) and \( H_2 \), consider an element \( ([5^w]_{2^{e-2}}, [a_2]_2, [a_3]_2) \in H \).
    We can express it as
    \[
    ([5^w]_{2^{e-2}}, [a_2]_2, [a_3]_2) = ([5^w]_{2^{e-2}}, [0]_2, [0]_2) \star ([1]_{2^{e-2}}, [a_2]_2, [a_3]_2).
    \]
    This shows that \( H = H_1 H_2 \).

    \item \textsf{Intersection of \( H_1 \) and \( H_2 \):} We have
    \[
    H_1 \cap H_2 = \{ ([5^{2^{e-5}}]_{2^{e-2}}, [0]_2, [0]_2)^i \mid i \in \{0, 1\} \} \subseteq Z(H),
    \]
    by Lemma \ref{vep} (3).

    \item \textsf{Centralization of Subgroups:} By Lemma \ref{vep} (3), \( H_1 \) centralizes \( H_2 \).
\end{itemize}
Thus, \( H \) is the internal central product of \( H_1 \) and \( H_2 \), and the proof is complete by \cite[Theorem 5.3]{zbMATH03724683}.
\qedhere
\end{enumerate}
\end{proof}

\begin{remark}
The software GAP was used to determine the structure of \( (\Aut(U_{2^e}, \cdot), \circ) \) for \( e \geq 5 \).
For small values of \( e \), we input the group \( V_e \) into the software, which enabled us to identify the corresponding group within the classification of finite groups.
Among the various possible representations, we chose the one based on the central product, as it provided a formulation that could be generalized to the broader case.
\end{remark}
From Proposition \ref{Prop2.3} and Proposition \ref{Prop2.7}, we can deduce the following:
\begin{itemize}
    \item When \( p \) is odd, we have
    \[
    (\mathcal{A}_{p^e}, \circ) = (\Aut(U_{p^e}, \cdot), \circ) \cong (U_{p^{e-1}(p-1)}, \cdot); 
    \]
    \item When \( e \leq 2 \), we have \( (\mathcal{A}_{2^e}, \circ) = (\{\Id\}, \circ) \);
    \item When \( e = 3 \), we have \( (\mathcal{A}_{8}, \circ) \cong (\mathbb{Z}/2\mathbb{Z}, +) \).
\end{itemize}
The following corollary determines the structure of \( \mathcal{A}_{2^e} \) when \( e \geq 4 \):
\begin{corollary}
When \( e \geq 4 \), we have \[ (\mathcal{A}_{2^e}, \circ) \cong (\mathbb{Z}/2^{e-4}\mathbb{Z} \times \mathbb{Z}/2\mathbb{Z} \times \mathbb{Z}/2\mathbb{Z}, +). \]
In particular, $(\mathcal{A}_{p^e}, \circ)$ is always an abelian group independently of the parity of \(p\).
\end{corollary}

\begin{proof}
When $e\geq 4$, by Theorem \ref{Theo2.6} and Proposition \ref{Prop2.7}, we know that \( \mathcal{A}_{2^e} \) is isomorphic to the subgroup \( W_e = \left\{([a_1]_{2^{e-2}}, [0]_2, [a_3]_2) \middle| [a_1]_{2^{e-2}} \in U_{2^{e-2}},[a_3]_2 \in \mathbb{Z}/2\mathbb{Z} \right\} \) of \( V_e \).

For any \( a_1, a_3, a_1', a_3' \in \mathbb{Z} \) with \( \gcd(a_1, 2) = \gcd(a_1', 2) = 1 \), we have
\[
\left([a_1]_{2^{e-2}}, [0]_2, [a_3]_2\right) \star \left([a_1']_{2^{e-2}}, [0]_2, [a_3']_2\right) = \left([a_1 a_1']_{2^{e-2}}, [0]_2, [a_3 + a_3']_2\right).
\]
From this and that \( (U_{2^{e-2}}, \cdot)\cong (\mathbb{Z}/2^{e-4}\mathbb{Z} \times \mathbb{Z}/2\mathbb{Z}, +) \), we conclude that \[ (\mathcal{A}_{2^e}, \circ) \cong (\mathbb{Z}/2^{e-4}\mathbb{Z} \times \mathbb{Z}/2\mathbb{Z} \times \mathbb{Z}/2\mathbb{Z}, +). \qedhere \]
\end{proof}
\begin{remark}\label{Rem2.15}
We denote by \(\Aut_{[-1]}(U_{2^e}, \cdot)\) (respectively \(\Aut_{[5]}(U_{2^e}, \cdot)\)) the automorphism group of \(U_{2^e}\) that fixes \([-1]\) (respectively \([5]\)).
We note that any automorphism fixing \([-1]\) (respectively \([5]\)) must also fix the corresponding groups \(\langle [-1] \rangle\) (respectively \(\langle [5] \rangle\)) pointwise.
Via the isomorphism $\Phi$ of Theorem \ref{Theo2.6} (1), we have:
\[
\Aut_{[-1]}(U_{2^e}, \cdot) = \left\{\chi_{2^e,(0,t_2,t_3)} \middle\mid ([t_3]_{2^{e-2}}, [t_2]_2) \in U_{2^{e-2}} \times \mathbb{Z}/2\mathbb{Z} \right\} \cong U_{2^{e-2}} \times \mathbb{Z}/2\mathbb{Z},
\]
and
\[
\Aut_{[5]}(U_{2^e}, \cdot) = \left\{ \chi_{2^e,(t_1, 0, 1)} \middle\mid \left[t_1/2^{e-3} \right]_2 \in \mathbb{Z}/2\mathbb{Z} \right\} \cong \mathbb{Z}/2\mathbb{Z}.
\]
In particular, we have that
\begin{align*}
    \Aut(U_{2^e}, \cdot) &= \left\{ \phi_{[-1]} \circ \phi_{[5]} \middle\mid \phi_{[-1]} \in \Aut_{[-1]}(U_{2^e}, \cdot), \phi_{[5]} \in \Aut_{[5]}(U_{2^e}, \cdot) \right\} \\
    &= \Aut_{[-1]}(U_{2^e}, \cdot) \circ \Aut_{[5]}(U_{2^e}, \cdot)
\end{align*}
and that \(\Aut(U_{2^e}, \cdot)\) is the internal semidirect product of \(\Aut_{[-1]}(U_{2^e}, \cdot)\) and \(\Aut_{[5]}(U_{2^e}, \cdot)\).
\end{remark}

\section{Characterizing the group of integers modulo \texorpdfstring{$p^e$}{}}
\subsection{Basic results about the monoid \texorpdfstring{$\mathbb{Z}/p^e\mathbb{Z}$}{}}
In the following, we adopt the convention that $s^0$ is defined as $[1]$ for any $s \in \mathbb{Z}/p^e\mathbb{Z}$, even when $s = [0]$.
Furthermore, we note that any element in $\mathbb{Z}/p^e\mathbb{Z}$ can be expressed as $[p]^u[r]$, where $\gcd(p,r) = 1$.
The following lemma provides a characterization of the conditions under which two elements in $\mathbb{Z}/p^e\mathbb{Z}$ are equal.

\begin{lemma}\label{Lem3.1}
Given \(u, u' \in \mathbb{N}\) and \(r, r' \in \mathbb{Z}\) with \(\gcd(p, r) = \gcd(p, r') = 1\), we have that \([p]^u[r] = [p]^{u'}[r']\) if and only if one of the following conditions holds:
\begin{enumerate}
    \item \(u, u' \geq e\), or
    \item \(u = u'\) and \([r]_{p^{e-u}}= [  r']_{p^{e-u}}\).
\end{enumerate}
\end{lemma}

\begin{proof}
The ``only if'' direction is clear.
For the ``if'' direction, since the implication is clear when $(1)$ is satisfied, we suppose that \([p]^u[r] = [p]^{u'}[r']\) and that \(u, u' \in \llbracket 0, e-1 \rrbracket\).

Assume, for contradiction, that \(u > u'\).
Thus, \(0 \leq e - u < e - u' \leq e\).
Multiplying both sides of the equation \([p]^u[r] = [p]^{u'}[r']\) by \([p]^{e - u}\), we obtain
\[
    [0] = [p]^e [r] = [p]^{u' - u + e} [r'].
\]
However, this leads to a contradiction because \([p]^{u' - u + e} [r'] \neq [0]\).
Therefore, \(u = u'\).

For \(u = u'\), the equation \([p]^u[r] = [p]^{u'}[r']\) is equivalent to \([ p^u r]=[  p^u r']\), which is equivalent to \([r]_{p^{e-u}} =[  r']_{p^{e-u}}\).
\end{proof}

\begin{lemma}\label{lem:ZpeZ_generators}
The cardinality of the minimal set of generators for $\mathbb{Z}/p^e\mathbb{Z}$ is $2$ when $p$ is odd and $3$ when $p = 2$.
More precisely, any set of generators for $\mathbb{Z}/p^e\mathbb{Z}$ of minimal cardinality is of the form $\lbrace [p]a\rbrace \cup G$ where
\begin{itemize}
    \item $a \in U_{p^e}$ and
    \item $G$ is a system of generators for $U_{p^e}$ with minimal cardinality.
\end{itemize}
\end{lemma}

\begin{proof}
Elements within $\mathbb{Z}/p^e\mathbb{Z}$ can be classified into two categories: invertible or non-invertible.
Products of invertible elements are invertible, while any product including a non-invertible element is non-invertible.
In particular, invertible elements do not include any powers of $[p]$ in their factorization as per Lemma \ref{Lem3.1}, while non-invertible elements are expressed using both non-trivial powers of $[p]$ and invertible elements.
Consequently, any multiplicative generating set for $\mathbb{Z}/p^e\mathbb{Z}$ must include, at the very least, the multiplicative generators for invertible elements and at least one non-invertible element.
Therefore, a system of generators of $\mathbb{Z}/p^e\mathbb{Z}$ with minimal cardinality is of the form $\lbrace [p]^ua \rbrace \cup G$ where $a\in U_{p^e}$, $u\in \mathbb{N}$ and $G$ is a system of generators for $U_{p^e}$ with minimal cardinality.

Let us suppose we have such a system.
Then there exists $x \in \mathbb{N}$ and $r \in U_{p^e}$ such that
$$[p] = ([p]^ua)^x r = [p]^{ux}a^xr.$$
This implies that $ux=1$ by Lemma \ref{Lem3.1} and thus $u=1$ since $u,x \in \mathbb{N}$.
The rest of the proof follows easily from here.
\end{proof}

\subsection{The structure of the automorphism group of \texorpdfstring{$\mathbb{Z}/p^e\mathbb{Z}$}{}}
We begin this section by characterizing the elements of \( \Aut(\mathbb{Z}/p^e\mathbb{Z},\cdot) \).
\begin{proposition} \label{prop:AutZpeZ}
The automorphism group of \( (\mathbb{Z}/p^e \mathbb{Z}, \cdot) \) is given by
\[
\Aut(\mathbb{Z}/p^e \mathbb{Z}, \cdot) = \left\lbrace \psi_{r, \phi} \mid r \in \mathbb{Z}, \gcd(r,p) = 1, \phi \in \mathcal{A}_{p^e} \right\rbrace,
\]
where for any \( r \in \mathbb{Z} \) with \( \gcd(r,p) = 1 \) and \( \phi \in \mathcal{A}_{p^e} \), the automorphism \( \psi_{r, \phi} \) is given by
\[
\psi_{r, \phi}([p]^u b) \coloneq [pr]^u \phi(b) \quad \text{for any} \ u \in \llbracket 0, e \rrbracket, \, b \in U_{p^e}.
\]
In particular, \( \mathcal{A}_{p^e} \) is precisely the set of automorphisms $(U_{p^e}, \cdot)$ that can be lifted to automorphism of \( (\mathbb{Z}/p^e \mathbb{Z}, \cdot) \).
\end{proposition}
\begin{proof}
Let \(\psi\) be an automorphism of \((\mathbb{Z}/p^e \mathbb{Z}, \cdot)\), and let \( c \in \mathbb{Z}/p^e \mathbb{Z} \).
We can express \( c \) as \( [p]^u [r] \), where \( u \in \llbracket 0, e \rrbracket \) and \( r \in \mathbb{Z} \) with \(\gcd(r, p) = 1\).
Since \( \psi \) is a homomorphism of \((\mathbb{Z}/p^e \mathbb{Z}, \cdot)\), we obtain:
\[
\psi(c) = \psi([p]^u [r]) = \psi([p])^u \tilde{\psi}([r]),
\]
where \( \tilde{\psi} \) is the automorphism of \( U_{p^e} \) induced by \( \psi \), as defined in Lemma \ref{Lem1.2}.

Moreover, by Lemma \ref{lem:ZpeZ_generators}, since \( \psi \) must map any system of generators to another system of generators and a zero divisor to a zero divisor, it follows that
\[
\psi([p]) = [pt], \quad \text{for some } [t] \in U_{p^e}.
\]
Now, consider a map \( \psi \) defined by
\[
\psi([p]^u [r]) = [pt]^u \tilde{\psi}([r])
\]
for all \( [p]^u [r] \in \mathbb{Z}/p^e\mathbb{Z} \), where \( \tilde{\psi} \in \Aut(U_{p^e}, \cdot) \) and $[t] \in U_{p^e}$.
We show that such a \(\psi\) is well-defined and injective if and only if \(\tilde{\psi} \in \mathcal{A}_{p^e}\), by examining the following sequence of equivalences:
\begin{align*}
    ([pt])^u \tilde{\psi}([r]) = ([pt])^{u'} \tilde{\psi}([r']), &\quad \forall u, u' \in \llbracket 0, e-1 \rrbracket, [r], [r'] \in U_{p^e} \\
    \iff u = u' \quad \text{and} \quad [t^u s]_{p^{e-u}} = [t^u s']_{p^{e-u}}, &\quad \forall u, u' \in \llbracket 0, e-1 \rrbracket, [r], [r'] \in U_{p^e} & \text{by Lemma \ref{Lem3.1}} \\
    \iff u = u' \quad \text{and} \quad [s]_{p^{e-u}} = [s']_{p^{e-u}}, &\quad \forall u, u' \in \llbracket 0, e-1 \rrbracket, [r], [r'] \in U_{p^e},
\end{align*}
where \(s\) and \(s'\) are integers such that \([s] = \tilde{\psi}([r])\) and \([s'] = \tilde{\psi}([r'])\), since \(\gcd(t, p) = 1\).
The latter assertion is equivalent to the following, if and only if \(\tilde{\psi} \in \mathcal{A}_{p^e}\), by Proposition \ref{Prop2.7}:
\begin{align*}
    u = u' \quad \text{and} \quad [r]_{p^{e-u}} = [r']_{p^{e-u}}, &\quad \forall u, u' \in \llbracket 0, e-1 \rrbracket, [r], [r'] \in U_{p^e} \\
    \iff [p]^u [r] = [p]^{u'} [r'], &\quad \forall u, u' \in \llbracket 0, e-1 \rrbracket, [r], [r'] \in U_{p^e}, \quad \text{by Lemma \ref{Lem3.1}}.
\end{align*}
This concludes the proof.
\end{proof}

Next, we define a group structure that we will later show is isomorphic to \((\Aut(\mathbb{Z}/p^e \mathbb{Z}, \cdot),\circ)\):
\begin{definitionlemma}
We define the \textsf{group \( (U_{p^{e-1}} \rtimes \mathcal{A}_{p^e}, \smallstar) \)} as the semidirect product of \( U_{p^{e-1}} \) and \( \mathcal{A}_{p^e} \).
This group is given by the set-theoretic product \( U_{p^{e-1}} \times \mathcal{A}_{p^e} \) equipped with the binary operation \( \smallstar \) defined by
\[
([r]_{p^{e-1}},\phi) \smallstar ([r']_{p^{e-1}},\phi') = ([r]_{p^{e-1}} \phi_{e-1} ([r']_{p^{e-1}}),\phi\circ\phi'),
\]
for all \( [r]_{p^{e-1}},[r']_{p^{e-1}} \in U_{p^{e-1}} \) and \( \phi, \phi' \in \mathcal{A}_{p^e} \).
\end{definitionlemma}

\begin{proof} Let $([r]_{p^{e-1}},\phi), ([r']_{p^{e-1}},\phi'), ([r'']_{p^{e-1}},\phi'')$ be elements of $U_{p^{e-1}} \rtimes \mathcal{A}_{p^e}$.
\begin{description}
\item[\sf{Associativity of $\smallstar$}]
\begin{align*}
   &\mathrel{\phantom{=}} \left(([r]_{p^{e-1}},\phi) \smallstar ([r']_{p^{e-1}},\phi')\right) \smallstar ([r'']_{p^{e-1}},\phi'')\\
   &= ([r]_{p^{e-1}}\phi_{e-1}([r']_{p^{e-1}}),\phi\circ\phi') \smallstar ([r'']_{p^{e-1}},\phi'') \\
    &= ([r]_{p^{e-1}}\phi_{e-1}([r']_{p^{e-1}})(\phi\circ\phi')_{e-1}([r'']_{p^{e-1}}), \phi \circ \phi' \circ \phi'') \\
    &=  ([r]_{p^{e-1}}\phi_{e-1}([r']_{p^{e-1}})\phi_{e-1}(\phi'_{e-1}([r'']_{p^{e-1}})), \phi \circ ( \phi' \circ \phi'') ), \quad \text{by Corollary \ref{theta}}\\
    &=  ([r]_{p^{e-1}}\phi_{e-1}([r']_{p^{e-1}}\phi'_{e-1}([r'']_{p^{e-1}})), \phi \circ ( \phi' \circ \phi'') ), \quad \text{by Definition-Lemma \ref{Deflem2.5}}\\
    &= ([r]_{p^{e-1}},\phi)\smallstar ([r']_{p^{e-1}}\phi'_{e-1}([r'']_{p^{e-1}}), \phi' \circ \phi'' )\\
     &= ([r]_{p^{e-1}},\phi)\smallstar \left( ([r']_{p^{e-1}},\phi')\smallstar ([r'']_{p^{e-1}},  \phi'' )\right).
\end{align*}

\item[{\sf{Identity element}}] We verify that the identity element is $([1]_{p^{e-1}}, \Id).$
\begin{align*}
    ([r]_{p^{e-1}},\phi) \smallstar ([1]_{p^{e-1}}, \Id) &= ([r]_{p^{e-1}}\Id_{e-1}([1]_{p^{e-1}}),\phi\circ\Id) = ([r]_{p^{e-1}},\phi), \text{ and}\\
    ([1]_{p^{e-1}}, \Id) \smallstar ([r]_{p^{e-1}},\phi)&=([1]_{p^{e-1}}\Id_{e-1}([r]_{p^{e-1}}),\Id\circ\phi) = ([r]_{p^{e-1}},\phi).
\end{align*}

\item[{\sf{Inverse element}}] The inverse element of $([r]_{p^{e-1}},\phi)$ is $(\phi^{-1}_{e-1}([r]_{p^{e-1}})^{-1},\phi^{-1})$.\\
Indeed, using Corollary \ref{theta}, we obtain
\begin{align*}
    (\phi^{-1}_{e-1}([r]_{p^{e-1}})^{-1},\phi^{-1}) \smallstar ([r]_{p^{e-1}},\phi) &= (\phi^{-1}_{e-1}([r]_{p^{e-1}})^{-1}\phi^{-1}_{e-1}([r]_{p^{e-1}}), \phi^{-1}\circ\phi) = ([1]_{p^{e-1}},\Id), \text{ and} \\
    ([r]_{p^{e-1}},\phi) \smallstar (\phi^{-1}_{e-1}([r]_{p^{e-1}})^{-1},\phi^{-1}) &= ([r]_{p^{e-1}} \phi_{e-1}(\phi^{-1}_{e-1}([r]_{p^{e-1}})^{-1}), \phi \circ \phi^{-1}) = ([r]_{p^{e-1}}[r]^{-1}_{p^{e-1}},\Id)\\ &= ([1]_{p^{e-1}},\Id). \qedhere
\end{align*}
\end{description}
\end{proof}

We can now deduce the main theorem of this paper, that $\Aut(\mathbb{Z}/p^e\mathbb{Z}, \cdot)$ is always isomorphic to the semi-direct product described in the previous definition.
\begin{theorem}\label{Theo3.3}
We have that
\[
\begin{array}{cccc}
    \Psi: & (\Aut(\mathbb{Z}/p^e\mathbb{Z}, \cdot),\circ) & \to & (U_{p^{e-1}} \rtimes \mathcal{A}_{p^e}, \smallstar) \\
    & \psi_{r, \phi} &\mapsto & ([r]_{p^{e-1}}, \phi)
\end{array}
\]
is a group isomorphism, where $\psi_{r,\phi}$ is as in Proposition \ref{prop:AutZpeZ}.
\end{theorem}

\begin{proof}
We prove that \(\Psi\) is well-defined and injective.
Let $r$ and $r'$ be integers coprime to $p$ and $\phi,\phi' \in \mathcal{A}_{p^e}$
such that \(\psi_{r, \phi} = \psi_{r', \phi'}\).
This is equivalent to having, for any \([s] \in U_{p^e}\), that \( \phi([s]) = \psi_{r, \phi}([s]) = \psi_{r', \phi'}([s]) = \phi'([s]) \), and \([pr]=\psi_{r, \phi}([p])  = \psi_{r', \phi'}([p]) = [pr'] \).
This in turn is equivalent to \(([r]_{p^{e-1}}, \phi) = ([r']_{p^{e-1}}, \phi')\).
Note that $\Psi$ is surjective by Proposition \ref{prop:AutZpeZ}.

To show that \(\Psi\) is a group homomorphism, consider \([p]^u t \in \mathbb{Z}/p^e \mathbb{Z}\) as above.
We compute:
\begin{align*}
    \mathrm{\phantom{=}} \psi_{r, \phi} \circ \psi_{r', \phi'}([p]^u t) &= \psi_{r, \phi}([pr']^u \phi'(t)) = \psi_{r, \phi} ([p]^u [r']^u \phi'(t)) \\
    &= [pr]^u \phi([r']^u \phi'(t)) \\
    &= ([p][r]\phi([r']))^u \phi(\phi'(t)) = \psi_{rs', \phi \circ \phi'}([p]^u t).
\end{align*}
where \( s'\) is an integer such that $[s']= \phi([r'])$.
Thus \(\Psi\) is indeed a group homomorphism.
\end{proof}

\begin{corollary} ~ \label{strape}
\begin{enumerate}
\item For $p$ an odd prime and $e=1$, we have:
\[
(\Aut(\mathbb{Z}/p\mathbb{Z}), \circ) \simeq (U_{p-1}, \cdot).
\]
In particular, $|\Aut(\mathbb{Z}/p\mathbb{Z})| = \varphi(p-1)$.

\item For \( p \) an odd prime and $e \geq 2$, we have:
\[
(\Aut(\mathbb{Z}/p^e \mathbb{Z}, \cdot), \circ) \simeq (\mathbb{Z}/p^{e-2}(p-1)\mathbb{Z} \rtimes U_{p^{e-1}(p-1)}, \ast),
\]
where the binary operation \(\ast\) is defined for any \( u, u', v, v' \in \mathbb{Z} \) with \( \gcd(v, p) = \gcd(v', p) = 1 \) by:
\[
([u]_{p^{e-2}(p-1)}, [v]_{p^{e-1}(p-1)}) \ast ([u']_{p^{e-2}(p-1)}, [v']_{p^{e-1}(p-1)}) = ([u + u'v]_{p^{e-2}(p-1)}, [vv']_{p^{e-1}(p-1)}).
\]
Using the notation of Proposition \ref{Prop2.3}, and Theorem \ref{Theo3.3}, this isomorphism can be explicitly constructed by choosing a generator \( a \) of \( U_{p^{e-1}} \) and mapping each automorphism \( \psi_{r, \chi_{p^e, t}} \) to \( ([s]_{p^{e-2}(p-1)}, [t]_{p^{e-1}(p-1)}) \), where \( s \) is an integer satisfying \(  [r]_{p^{e-1}} =a^s \).\\
In particular, \[|\Aut(\mathbb{Z}/p^e \mathbb{Z}, \cdot)|= (p^{e-2}(p-1))^2 \varphi (p-1)\]
where \(\varphi\) denotes Euler's totient function.

\item We have $$(\Aut(\mathbb{Z}/8 \mathbb{Z}, \cdot), \circ)\simeq(\mathbb{Z}/2\mathbb{Z} \times \mathbb{Z}/2\mathbb{Z}, +).$$

\item
When \( p = 2 \) and \( e \geq 4 \), we have the isomorphism:
\[
(\Aut(\mathbb{Z}/2^e \mathbb{Z}, \cdot), \circ) \simeq \big( \mathbb{Z}/2\mathbb{Z} \times \mathbb{Z}/2\mathbb{Z} \times (\mathbb{Z}/2^{e-3}\mathbb{Z} \rtimes U_{2^{e-2}}), \# \big),
\]
where the binary operation \( \mathbin{\#} \) is defined for any \( u, s, v, t, u', s', v', t' \in \mathbb{Z} \) with \( \gcd(t, 2) = \gcd(t', 2) = 1 \) by:
\begin{align*}
    &\mathrel{\phantom{=}} ([u]_2, [s]_2, ([v]_{2^{e-3}}, [t]_{2^{e-2}})) \mathbin{\#} ([u']_2, [s']_2, ([v']_{2^{e-3}}, [t']_{2^{e-2}})) \\
    &= ([u + u']_2, [s + s']_2, ([v + tv']_{2^{e-3}}, [tt']_{2^{e-2}})).
\end{align*}
Using the notation from Theorem \ref{Theo2.6}, Proposition \ref{Prop2.7}, and Theorem \ref{Theo3.3}, we can explicitly construct this isomorphism by mapping each automorphism \( \psi_{r, \chi_{p^e, (s, 0, t)}} \) to
\[
([u]_2, [s]_2, ([v]_{2^{e-3}}, [t]_{2^{e-2}})),
\]
where \( u, v \) are integers satisfying
\[
[r]_{2^{e-1}} = [-1]_{2^{e-1}}^u [5]_{2^{e-1}}^v.
\]
In particular, for \( e = 4 \), we obtain:
\[
(\Aut(\mathbb{Z}/16 \mathbb{Z}, \cdot), \circ) \simeq (\mathbb{Z}/2\mathbb{Z} \times \mathbb{Z}/2\mathbb{Z} \times \mathbb{Z}/2\mathbb{Z} \times \mathbb{Z}/2\mathbb{Z}, +).
\]
In general,
\[
|\Aut(\mathbb{Z}/2^e \mathbb{Z}, \cdot)| = 2^{2(e-2)}.
\]
\end{enumerate}
\end{corollary}

\begin{proof} ~
\begin{enumerate}
\item This follows from Proposition \ref{Prop2.3} and Theorem \ref{Theo3.3}.
\item This follows from Proposition \ref{Prop2.3} and Theorem \ref{Theo3.3}. 
\item By Theorem \ref{Theo3.3}, $(\Aut(\mathbb{Z}/8 \mathbb{Z}, \cdot), \circ) \simeq (U_4 \rtimes \mathcal{A}_{8}, \smallstar)$.
But $U_4 = \{[1],[3]\} \simeq \mathbb{Z}/2\mathbb{Z} \simeq \mathcal{A}_8$, where by Proposition \ref{Prop2.7} $\phi \in \mathcal{A}_8$ is defined by having fixed points $[1]$ and $[5]$, and either swapping $[3]$ and $[7]$ or leaving them fixed.
Thus $\phi_2 = \Id_{U_4}$, and so the group action $\smallstar$ in the semidirect product makes it into the direct product $\mathbb{Z}/2\mathbb{Z} \times \mathbb{Z}/2\mathbb{Z}$.
\item By Theorem \ref{Theo2.6}, Proposition \ref{Prop2.7}, and Theorem \ref{Theo3.3}, for any \( [r]_{2^{e-1}}, [r']_{2^{e-1}} \in U_{2^{e-1}} \), we can express them as:
\[
[r]_{2^{e-1}} = [-1]^{u}_{2^{e-1}} [5]^{v}_{2^{e-1}}, \quad [r']_{2^{e-1}} = [-1]^{u'}_{2^{e-1}} [5]^{v'}_{2^{e-1}},
\]
for some \( u, u' \in \{0,1\} \) and \( v, v' \in \llbracket 0, 2^{e-3} - 1\rrbracket \).  \\
Now, consider \( \phi = \chi_{2^e, (s,0, t)} \) and \( \phi' = \chi_{2^e, (s',0, t')} \) in \( \mathcal{A}_{2^e} \).
By Remark \ref{Ape}, we have
\[
\phi_{e-1} = \chi_{2^{e-1}, (0,0, t)}.
\]
Thus, we compute:
\begin{align*}
    [r]_{2^{e-1}} \phi_{e-1}([r']_{2^{e-1}})
    &= [-1]_{2^{e-1}}^{u} [5]^{v}_{2^{e-1}} \phi_{e-1}([-1]_{2^{e-1}}^{u'} [5]^{v'}_{2^{e-1}}) \\
    &= [-1]_{2^{e-1}}^{u} [5]^{v}_{2^{e-1}} \chi_{2^{e-1}, (0,0, t)}([-1]_{2^{e-1}}^{u'} [5]^{v'}_{2^{e-1}}) \\
    &= [-1]_{2^{e-1}}^u [5]^{v}_{2^{e-1}} \cdot ([-1]_{2^{e-1}}^{u'}) \cdot ([5]^{t}_{2^{e-1}})^{v'} \\
    &= [-1]_{2^{e-1}}^{u+u'} [5]^{v + t v'}_{2^{e-1}}.
\end{align*}
Furthermore, the composition in \( \mathcal{A}_{2^e} \) satisfies:
\[
\chi_{2^e, (s,0, t)} \circ \chi_{2^e, (s',0, t')} = \chi_{2^e, (s + s',0, tt')}. \qedhere
\]
\end{enumerate}
\end{proof}

\begin{remark} ~
\begin{itemize}
    \item We note that for any automorphism of $\mathbb{Z}/8\mathbb{Z}$, we have $\phi([4]) = [4]$.
    \item $(\Aut(\mathbb{Z}/p^e \mathbb{Z}, \cdot), \circ)$ is a non commutative group when $p$ is odd or, $p=2$ and $e\geq 5$.
    \item We denote by \(\Aut_{U_{p^e}} (\mathbb{Z}/p^e \mathbb{Z}, \cdot)\) and \(\Aut_{[p]} (\mathbb{Z}/p^e \mathbb{Z}, \cdot)\) the automorphism groups of \(\mathbb{Z}/p^e \mathbb{Z}\) that fix \(U_{p^e}\) and \([p]\) pointwise, respectively.
    
    The map \(\Psi\) induces the following isomorphisms:
    \begin{align*}
        \Aut_{U_{p^e}} (\mathbb{Z}/p^e \mathbb{Z}, \cdot) &= \left\{ \psi_{r, \Id} \mid r \in U_{p^{e-1}} \right\} \cong U_{p^{e-1}} \\
         \shortintertext{and}
        \Aut_{[p]} (\mathbb{Z}/p^e \mathbb{Z}, \cdot) &= \left\{ \psi_{1, \phi} \mid \phi \in \Aut(U_{p^e}, \cdot) \right\} \cong \mathcal{A}_{p^e}.
    \end{align*}
    Thus we have by Theorem \ref{Theo3.3} that \(\Aut(\mathbb{Z}/p^e \mathbb{Z}, \cdot)\) is the internal semidirect product of \(\Aut_{U_{p^e}} (\mathbb{Z}/p^e \mathbb{Z}, \cdot)\) and \(\Aut_{[p]} (\mathbb{Z}/p^e \mathbb{Z}, \cdot)\).
    In particular, we have:
    \begin{align*}
        &\mathrel{\phantom{=}} \Aut_{U_{p^e}}(\mathbb{Z}/p^e \mathbb{Z}, \cdot) \circ \Aut_{ [p] }(\mathbb{Z}/p^e \mathbb{Z}, \cdot) \\
        &= \left\{\psi_{U_{p^e}} \circ \psi_{[p]} \middle| \psi_{U_{p^e}} \in \Aut_{U_{p^e}}(\mathbb{Z}/p^e \mathbb{Z}, \cdot), \psi_{[p]} \in \Aut_{ [p] }(\mathbb{Z}/p^e \mathbb{Z}, \cdot)\right\} \\
        &=\Aut(\mathbb{Z}/p^e \mathbb{Z}, \cdot).
    \end{align*}
\end{itemize}
\end{remark}


\end{document}